\documentclass[12pt,a4paper]{article}
\usepackage{amsmath,amstext,amssymb,amscd, color}
\usepackage{graphicx}

\usepackage[english]{babel}

\newcommand{\Xcomment}[1]{}

\newtheorem{theorem}{Theorem}[section]
\newtheorem{lemma}[theorem]{Lemma}
\newtheorem{corollary}[theorem]{Corollary}

\newtheorem{prop}[theorem]{Proposition}

\newcommand{\SEC}[1]{\ref{sec:#1}}  
\newcommand{\SSEC}[1]{\ref{ssec:#1}}  

\makeatletter \@addtoreset{equation}{section} \makeatother

\newenvironment{proof}{\noindent{\bf Proof}~}%
{\hfill$\qed$\medskip}

\def\qed{ \ \vrule width.1cm height.3cm depth0cm}

\newenvironment{numitem1}{\refstepcounter{equation}\begin{enumerate}%
\item[(\thesection.\arabic{equation})]}{\end{enumerate}}

\newcommand{\refeq}[1]{(\ref{eq:#1})}  

 \makeatletter
\renewcommand{\section}{\@startsection{section}{1}{0pt}%
{-3.5ex plus -1ex minus -.2ex}{2.3ex plus .2ex}%
{\normalfont\Large}}
 \makeatother

 \makeatletter
\renewcommand{\subsection}{\@startsection{subsection}{2}{0pt}%
{-3.0ex plus -1ex minus -.2ex}{-1.5ex plus .2ex}%
{\normalfont\large\bf}}
 \makeatother

 \Xcomment{
 \makeatletter
\renewcommand{\subsection}{\@startsection{subsection}{2}{0pt}%
{-3.0ex plus -1ex minus -.2ex}{1.5ex plus .2ex}%
{\normalfont\large\bf}}
  }
 \makeatother

\def\Rset{{\mathbb R}}
\def\Zset{{\mathbb Z}}

\def\Ascr{{\cal A}}
\def\Bscr{{\cal B}}
\def\Cscr{{\cal C}}
\def\Dscr{{\cal D}}

\def\Fscr{{\cal F}}
\def\Gscr{{\cal G}}

\def\Kscr{{\cal K}}

\def\Mscr{{\cal M}}
\def\Nscr{{\cal N}}
\def\Oscr{{\cal O}}

\def\Sscr{{\cal S}}

\def\Wscr{{\cal W}}

\def\frakS{\mathfrak{S}}

\def\frakB{\mathfrak{B}}
\def\frakD{\mathfrak{D}}

\def\Inver{{\rm Inv}}

\def\tilde{\widetilde}
\def\hat{\widehat}
\def\bar{\overline}
\def\eps{\epsilon}

\def\bfC{{\bf C}}
\def\bfQ{{\bf Q}}
\def\bfS{{\bf S}}

\def\bfW{{\bf W}}
\def\bfCsym{\bf{sym\mbox{-}C}}
\def\bfSsym{\bf{sym\mbox{-}S}}
\def\bfWsym{\bf{sym\mbox{-}W}}
\def\bfQsym{\bf{sym\mbox{-}Q}}
\def\bfMsym{\bf{sym\mbox{-}M}}
\def\bfDsym{\bf{sym\mbox{-}D}}

\def\Tmin{T^{\rm min}}

\def\Poss{{\bf PS}}
\def\Possym{{\bf PS}^{\rm sym}}
\def\Posw{{\bf PW}}

\def\Zfr{{Z^{\,\rm fr}}}
\def\Zrear{{Z^{\,\rm rear}}}

\def\Zpfr{Z'^{\,\rm fr}}
\def\Zprear{Z'^{\,\rm rear}}

\def\Cfr{C^{\,\rm fr}}
\def\Cpfr{C'^{\,\rm fr}}
\def\Crear{C^{\,\rm rear}}
\def\Castfr{C^{\ast\,\rm fr}}
\def\Castrear{C^{\ast\,\rm rear}}

\def\frakDst{\frakD^{\,\rm st}}
\def\frakDant{\frakD^{\,\rm ant}}
\def\frakDfr{\frakD^{\,\rm fr}}
\def\frakDrear{\frakD^{\,\rm rear}}

\def\frakDastst{\frakD^{\ast\,\rm st}}
\def\frakDastant{\frakD^{\ast\,\rm ant}}

\def\Qst{Q^{\,\rm st}}
\def\Qant{Q^{\,\rm ant}}
\def\Ist{I^{\,\rm st}}
\def\Iant{I^{\,\rm ant}}

\def\Inver{{\rm Inv}}

\def\Gammas{\Gamma^{\,\rm s}}
\def\Gammaw{\Gamma^{\,\rm w}}
\def\Gammasym{\Gamma^{\,\rm sym}}

\def\Klow{K^{\,\rm low}}
\def\Kup{K^{\,\rm up}}
\def\Kuppl{K^{\,\rm up+}}
\def\Zlow{Z^{\,\rm low}}
\def\Zup{Z^{\,\rm up}}
\def\Tlow{T^{\,\rm low}}
\def\Tup{T^{\,\rm up}}
\def\lambdalow{\lambda^{\,\rm low}}
\def\lambdaup{\lambda^{\,\rm up}}
\def\Llow{L^{\,\rm low}}
\def\Lup{L^{\,\rm up}}
\def\Clow{C^{\,\rm low}}
\def\Cup{C^{\,\rm up}}
\def\Cplow{C'^{\,\rm low}}

\def\Zplow{Z'^{\,\rm low}}
\def\Zpup{Z'^{\,\rm up}}
\def\Qplow{Q'^{\,\rm low}}
\def\Qpup{Q'^{\,\rm up}}

\def\deltain{\delta^{\,\rm in}}

\def\dsymm{{\lozenge}}

\begin{document}

 \title{Flips in symmetric separated set-systems}

 \author{Vladimir I.~Danilov
 \thanks{Central Institute of Economics and
Mathematics of the RAS, 47, Nakhimovskii Prospect, 117418 Moscow, Russia;
email: danilov@cemi.rssi.ru.}
 \and
Alexander V.~Karzanov
\thanks{Central Institute of Economics and Mathematics of
the RAS, 47, Nakhimovskii Prospect, 117418 Moscow, Russia; email:
akarzanov7@gmail.com. Corresponding author. }
  \and
Gleb A.~Koshevoy
\thanks{The Institute for Information Transmission Problems of
the RAS, 19, Bol'shoi Karetnyi per., 127051 Moscow, Russia, and HSE University;
email:koshevoyga@gmail.com. Supported in part by Laboratory of Mirror Symmetry
NRU HSE, RF Government grant, ag. No. 14.641.31.0001.}
  }

\date{}

 \maketitle

\begin{abstract}

For a positive integer $n$, a collection $\Sscr$ of subsets of
$[n]=\{1,\ldots,n\}$ is called \emph{symmetric} if $X\in \Sscr$ implies
$X^\ast\in\Sscr$, where $X^\ast:=\{i\in [n]\colon n-i+1\notin X\}$ (the
involution $\ast$ was introduced by Karpman). Leclerc and Zelevinsky showed
that the set of maximal strongly (resp. weakly) separated collections in
$2^{[n]}$ is connected via flips, or mutations, ``in the presence of six (resp.
four) witnesses''. We give a symmetric analog of those results, by  showing
that each maximal symmetric strongly (weakly) separated collection in $2^{[n]}$
can be obtained from any other one by a series of special symmetric local
transformations, so-called \emph{symmetric flips}. Also we establish the
connectedness via symmetric flips for the class of maximal symmetric
$r$-separated collections in $2^{[n]}$ when $n,r$ are even (where sets
$A,B\subseteq [n]$ are called $r$-separated if there are no elements $i_0<i_1<
\cdots <i_{r+1}$ in $[n]$ which alternate in $A\setminus B$ and $B\setminus
A$). This is related to a symmetric version of higher Bruhat orders.

These results are obtained as consequences of our study of related geometric
objects: symmetric rhombus and combined tilings and symmetric cubillages.
  \medskip

   \end{abstract}

\baselineskip=15pt
\parskip=1pt

\section{Introduction}  \label{sec:intr}

We fix a positive integer $n$ and interpret the elements of the set
$[n]:=\{1,2,\ldots,n\}$ as \emph{colors}. For each $i\in[n]$, the color $n-i+1$
is regarded as \emph{complementary} to $i$ and denoted as $i^\circ$. Following
Karpman~\cite{karp}, for  $X\subseteq[n]$, define the set
 \begin{equation} \label{eq:Xast}
 X^\ast:=\{i\in[n]\colon i^\circ\notin X\}.
 \end{equation}
One can see that $(X^\ast)^\ast=X$; so the relation $\ast$ gives an involution
on the set $2^{[n]}$ of all subsets of $[n]$. We say that the sets $X$ and
$X^\ast$ are \emph{$\ast$-symmetric}, or, simply, \emph{symmetric}, to each
other. When $X$ coincides with $X^\ast$, it is called \emph{self-symmetric}.
Accordingly, a collection $\Sscr\subseteq 2^{[n]}$ is called \emph{symmetric}
if $X\in\Sscr$ implies $X^\ast\in\Sscr$.

Recently, extending a result in~\cite{karp}, it was proved in~\cite{DKK4} that
symmetric strongly, weakly and chord separated collections $\Sscr$ in $2^{[n]}$
possess the property of \emph{purity}, which means that if $\Sscr$ is maximal
by inclusion (among all collections of the given type), then it is maximal by
size. This matches the purity behavior for usual strongly, weakly and chord
separated collections in $2^{[n]}$ (proved in~\cite{LZ}, \cite{DKK1},
\cite{gal}, respectively).  Recall that
  \begin{itemize}
\item Sets $A,B\subseteq[n]$ are called
\emph{strongly separated} (from each other) if there are no three elements
$i<j<k$ of $[n]$ such that one of $A-B$ and $B-A$ contains $i,k$, and the other
contains $j$.
\item
Sets $A,B\subseteq[n]$ are called \emph{chord separated} (or
\emph{2-separated}) if there are no elements $i<j<k<\ell$ of $[n]$ such that
one of $A-B$ and $B-A$ contains $i,k$, and the other contains $j,\ell$.
\item
Sets $A,B\subseteq[n]$ are called \emph{weakly separated} if they are chord
separated and the additional condition holds: if $A$ surrounds $B$ and
$B-A\ne\emptyset$ then $|A|\le|B|$, and if $B$ surrounds $A$ and
$A-B\ne\emptyset$ then $|B|\le|A|$.
  \end{itemize}

\noindent Accordingly,  a collection $\Ascr\subseteq 2^{[n]}$ of subsets of
$[n]$ is called strongly (chord, weakly) separated if any two members of
$\Ascr$ are strongly (resp. chord, weakly) separated.

(Hereinafter, for sets $A,B\subseteq[n]$, $|A|$ is the size (the number of
elements) of $A$; $A-B$ denotes the set difference $\{i\colon A\ni i\notin
B\}$; we write $A<B$ if the maximal element $\max(A)$ of $A$ is smaller than
the minimal element $\min(B)$ of $B$, letting $\min(\emptyset):=\infty$ and
$\max(\emptyset):=-\infty$; and we say that $A$ \emph{surrounds} $B$ if
$\min(A-B)<\min(B-A)$ and $\max(A-B)>\max(B-A)$.)

For brevity, in what follows we refer to strongly, weakly, and chord separated
collections as \emph{s-}, \emph{w-}, and \emph{c-collections}, respectively.
The sets (classes) of maximal collections among those in $2^{[n]}$ are denoted
by $\bfS_n$, $\bfW_n$, and $\bfC_n$, respectively. Their $\ast$-symmetric
counterparts are denoted by $\bfSsym_n$, $\bfWsym_n$, and $\bfCsym_n$,
respectively.

It is well-known that each of $\bfS_n,\bfW_n,\bfC_n$ is extended to a poset
with a unique minimal and a unique maximal  elements; in particular, such
posets are connected. Here the poset structures arise when the collections
forming these classes are linked by binary relations, or ``edges'', where each
edge corresponds to a local transformation, called a \emph{flip} (or
mutation), which turns one collection into a ``neighboring'' collection in this
class. See~\cite{LZ,gal} for details.

Leclerc and Zelevinsky~\cite{LZ} revealed the poset structures on $\bfS_n$ and
$\bfW_n$ (establishing flips ``in the presence of six and four witnesses'',
respectively), and the poset structure on $\bfC_n$ was demonstrated by
Galashin~\cite{gal}.

In this paper, we introduce local transformations, called \emph{symmetric
flips}, on the members of $\bfSsym_n$, $\bfWsym_n$, and $\bfCsym_n$ (with $n$
even in the last case), and show that such flips endow these classes with the
structure of connected graphs (though not necessarily posets). In other words,
for any two maximal symmetric collections in each class, we can obtain one from
the other by a series of flips within this class.

Extending some of these results, we further introduce and study symmetric flips
for more sophisticated classes, namely, ones formed by size-maximal symmetric
strongly $r$-separated collections in $2^{[n]}$. Here for an integer $r\ge 1$,
sets $A,B\subseteq [n]$ are called $r$-\emph{separated} if there are no
elements $i_0<i_1< \cdots <i_{r+1}$ of $[n]$ alternating in $A-B$ and $B-A$.
Accordingly, a collection $\Sscr\subseteq 2^{[n]}$ is called strongly
$r$-separated if any two of its members are such. Usually, when speaking of
such sets and collections, the adjective ``strong'' will be omitted for
brevity. (So s- and c-collections are just 1- and 2-separated ones,
respectively.)

An essential part of our study will be focused on the case when both $n$ and
$r$ are even, yielding a generalization for symmetric c-collections with $n$
even. We show that in this case the symmetric flip structure forms a connected
poset. Moreover, we explain that this poset is associated with a structure on
$\binom{[n]}{r+1}$ that can be interpreted as a symmetric version of higher
Bruhat orders, which is related to \emph{type C} parameterized by $(n,r+1)$
(where $\binom{[n]}{m}$ is formed by the $m$-element subsets of $[n]$). (Recall
that the notion of higher Bruhat orders was introduced by Manin and
Schekhtman~\cite{MS} as a generalization of the classical weak Bruhat order on
permutations.) A wider discussion on higher Bruhat orders of types B and C will
appear in the forthcoming paper~\cite{DKK6}.

As is shown in Galashin and Postnikov~\cite{GP}, the purity behavior does not
continue to hold when $\min\{r,n-r\}\ge 3$; namely, there exist maximal by
inclusion $r$-separated collection in $2^{[n]}$ which are not maximal by size.
A similar behavior takes place for symmetric $r$-separated collections as well.

In light of this, we write $s_{n,r}$ for the maximal possible size $|\Sscr|$
among \emph{all} (not necessarily symmetric) $r$-separated collections $\Sscr$
in $2^{[n]}$, and refer to such collections as \emph{size-maximal}. The class
of $r$-separated collections in $2^{[n]}$ which are size-maximal and
simultneously $\ast$-symmetric is denoted as $\bfSsym_{n,r}$. (Note that a
priory $\bfSsym_{n,r}$ may be empty; e.g. this happens when $n$ is odd and
$r=1$.) When both $n,r$ are even, we show that $\bfSsym_{n,r}\ne\emptyset$,
introduce symmetric flips for the members of $\bfSsym_{n,r}$ and establish a
poset structure on $\bfSsym_{n,r}$ (as mentioned above).

In order to obtain the above-mentioned results, we essentially attract
geometric methods. We are based on the nice bijections between (i) maximal
s-collections and \emph{rhombus tilings} (see~\cite[Th.~1.6]{LZ} where the
language of pseudo-line arrangements, dual to rhombus tilings, is used), (ii)
maximal w-collections and \emph{combined tilings} (see~\cite{DKK2}), and (iii)
between size-maximal $r$-separated collections and fine zonotopal tilings, or
\emph{cubillages}, on cyclic zonotopes of dimension $r+1$. In these cases, the
sets in each collection $\Sscr$ as above are encoded the vertices of the
corresponding geometric object. (For various aspects of cubillages, see
survey~\cite{DKK3}.)

It turns out that similar relationships between set-systems and geometric
objects take place in the symmetric versions as well (for symmetric s-, w-, and
c-cases, see~\cite{DKK4}). Using such correspondences, we study symmetric
tilings and cubillages, introduce symmetric flips on them, and examine the
obtained flip graph in each case.
 \medskip

This paper is organized as follows. Section~\SEC{prelim} reviews basic facts on
rhombus and combined tilings needed to us. Section~\SEC{scoll} is devoted to
flips in symmetric s-collections (yielding Theorem~\ref{tm:all-s-flips}), and
Subsection~\SSEC{w-n-even} to flips in symmetric w-collections (yielding
Corollary~\ref{cor:all-w-flips}), both dealing with an even number $n$ of
colors. The case of symmetric s- and w-collections with $n$ odd is considered
in Subsection~\SSEC{sw-n-odd}. Both Sections~\SEC{scoll} and~\SEC{wcoll} are
well illustrated to facilitate understanding.

Then we proceed to a study of the flip structure for symmetric $r$-separated
collections in $2^{[n]}$ and symmetric $d$-dimensional cubillages on the cyclic
zonotope $Z(n,d=r+1)$. Section~\SEC{rsep-cubil} gives necessary definitions and
reviews  known results on important ingredients of cubillages that will be used
later, in particular, \emph{membranes} (subcomplexes of a cubillage that admit
bijective projections to cubillages of the previous dimension).  In
Section~\SEC{symm-r-even} we show that for $n,r$ even, the symmetric flip
structure forms a poset with one minimal and one maximal elements (see
Theorem~\ref{tm:Gammasym} and Corollary~\ref{cor:flip-rsepar}). This leads to a
symmetric higher Bruhat order (of type C parameterized by $(n,d=r+1)$), which
is discussed in Section~\SEC{Bruhat} and summarized in
Theorem~\ref{tm:BruhatC}. The main part of the paper finishes with
Section~\SEC{gen-theorems} which presents assertions of a general character,
showing that a symmetric cubillage contains a symmetric membrane
(Theorem~\ref{tm:cub-membr}), and that any symmetric cubillage can be lifted as
a membrane in a symmetric cubillage of the next dimension
(Theorem~\ref{tm:membr-cub}). As a consequence, any cyclic zonotope $Z(n,d)$
with $n$ even has at least one symmetric cubillage
(Corollary~\ref{cor:zon-sym-cub}).

Three additional settings are discussed in appendixes to this paper. Appendix~A
considers symmetric $r$-separated collections in $2^{[n]}$ when $n$ is even but
$r$ is odd, while Appendix~B deals with $n$ odd and $r$ even. In its turn,
Appendix~C is devoted to symmetric \emph{weakly $r$-separated} set-systems. In
each of these three cases, we specify the definitions of appropriate symmetric
flips and raise conjectures whose validity would imply the connectedness of the
corresponding flip graphs.


\section{Preliminaries}  \label{sec:prelim}

In this section we recall (or specify) the definitions of zonogon, rhombus and
combined tilings and flips in them, and review some basic properties of these
objects that we will use later on.

As before, let $n\in\Zset_{>0}$.
An \emph{interval} in $[n]$ is a set of
the form $\{a,a+1,\ldots,b\}\subseteq[n]$, denoted as $[a..b]$ (so
$[n]=[1..n]$). For disjoint subsets $A$ and $\{a,\ldots,b\}$ of $[n]$, we
will use the abbreviated notation $Aa\ldots b$ for $A\cup\{a,\ldots,b\}$, and
write $A-c$ for $A-\{c\}$ when $c\in A$.\medskip

\noindent\textbf{Zonogon.} ~Let $\Xi$ be a set of $n$ vectors
$\xi_i=(x_i,y_i)\in \Rset^2$ such that
  \begin{numitem1} \label{eq:xi-x-y}
  $x_1<\cdots <x_n$ and $y_i=1-\delta_i$, $i=1,\ldots,n$,
    \end{numitem1}
where each $\delta_i$ is a sufficiently small positive real. For our purposes,
we additionally assume that
  \begin{numitem1} \label{eq:additional}
(i) $\Xi$ satisfies the \emph{strict concavity} condition:  for any $i<j<k$,
there exist $\lambda,\lambda'\in \Rset_{>0}$ such that $\lambda+\lambda'>1$ and
$\xi_j=\lambda\xi_i+\lambda'\xi_k$; ~and (ii) the vectors in $\Xi$ are
$\Zset_2$-independent (i.e., all 0,1-combinations of these vectors are
different).
  \end{numitem1}

The \emph{zonogon} generated by $\Xi$ is the Minkowski sum
of segments $[0,\xi_i]$:
  $$
Z=Z(\Xi):=\{\lambda_1\xi_1+\ldots+ \lambda_n\xi_n\colon \lambda_i\in\Rset,\;
0\le\lambda_i\le 1,\; i=1,\ldots,n\}.
  $$
The choice of $\Xi$ is usually not important to us (subject
to~\refeq{xi-x-y},\refeq{additional}), and we may denote $Z$ as $Z(n,2)$. Each
subset $X\subseteq[n]$ is identified with the point $\sum_{i\in X} \xi_i$ in
$Z$; so, due to~\refeq{additional}(ii), different subsets are identified with
different points.

Besides $\xi_1,\ldots,\xi_n$, we will often use the vectors $\eps_{ij}:=\xi_j-\xi_i$ for
$1\le i<j\le n$.
  \medskip

\noindent\textbf{Rhombus and combined tilings.} ~A tiling of these sorts is a
subdivision of the zonogon $Z=Z(\Xi)$ into convex polygons specified below and
called \emph{tiles}. Any two intersecting tiles share a common vertex or edge,
and each edge of the boundary of $Z$ belongs to exactly one tile. We associate
to a tiling $T$ the planar graph $(V_T,E_T)$ whose vertex set $V_T$ and edge
set $E_T$ are formed by the vertices and edges occurring in tiles. Each vertex
is (a point identified with) a subset of $[n]$. And each edge is a line segment
viewed as a parallel translation of either $\xi_i$ or $\eps_{ij}$ for some
$i<j$.  In the former case, it is called an edge of \emph{type} or \emph{color}
$i$, or an $i$-\emph{edge}, and in the latter case, an edge of \emph{type}
$ij$, or an $ij$-\emph{edge}. These edges are directed accordingly. In
particular, the \emph{left} (\emph{right}) \emph{boundary} of $T$ forms the
directed path $(v_0,e_1,v_1,\ldots, e_n, v_n)$ in which each vertex $v_i$
represents the interval $[i]$ (resp. the interval $[(n+1-i)..n]$).
\medskip

In a \emph{rhombus tiling}, each tile $\tau$ is a parallelogram with edges of
types $i$ and $j$. For brevity we refer to $\tau$  as a \emph{rhombus}, or an
$ij$-\emph{rhombus} (where $i<j$),  and denote it as $\lozenge=\lozenge(X|ij)$,
where $X$ is its bottommost vertex, or simply the \emph{bottom} of $\lozenge$.
\smallskip

In a \emph{combined tiling}, or a \emph{combi} for short, there are three sorts
of tiles, namely, $\Delta$-tiles, $\nabla$-tiles, and lenses, as illustrated in
the picture below.

\vspace{-0.2cm}
\begin{center}
\includegraphics{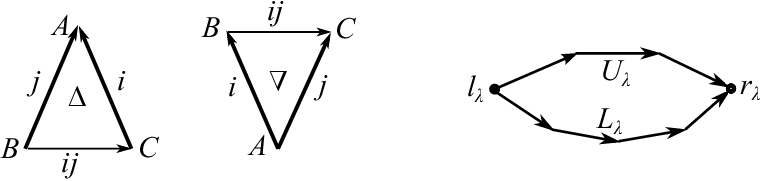}
\end{center}
\vspace{-0.2cm}

I. A $\Delta$-\emph{tile} ($\nabla$-\emph{tile}) is a triangle with vertices
$A,B,C\subseteq[n]$ and edges $(B,A),(C,A),(B,C)$ (resp. $(A,C),(A,B),(B,C)$)
of types $j$, $i$ and $ij$, respectively, where $i<j$. We denote this tile as
$\Delta(A|BC)$ (resp. $\nabla(A|BC)$).

II. In a \emph{lens} $\lambda$, the boundary is formed by two directed paths
$U_\lambda$ and $L_\lambda$, with at least two edges in each, having the same
beginning (or \emph{left}) vertex $\ell_\lambda$ and the same end
(\emph{right}) vertex $r_\lambda$. The \emph{upper boundary}
$U_\lambda=(v_0,e_1,v_1,\ldots,e_p,v_p)$ is such that $v_0=\ell_\lambda$,
$v_p=r_\lambda$, and $v_k=Xi_k$ for $k=0,\ldots,p$, where $p\ge 2$,
$X\subset[n]$ and $i_0<i_1<\cdots <i_p$ (so $k$-th edge $e_k$ is of type
$i_{k-1}i_k$). And the \emph{lower boundary}
$L_\lambda=(u_0,e'_1,u_1,\ldots,e'_q,u_q)$ is such that $u_0=\ell_\lambda$,
$u_q=r_\lambda$, and $u_m=Y-j_m$ for $m=0,\ldots,q$, where $q\ge 2$,
$Y\subseteq [n]$ and $j_0>j_1>\cdots>j_q$ (so $m$-th edge $e'_m$ is of type
$j_mj_{m-1}$). Then $Y=Xi_0j_0=Xi_pj_q$, implying $i_0=j_q$ and $i_p=j_0$. Note
that $X$ as well as $Y$ need not be a vertex in a combi. Due
to~\refeq{additional}(i), $\lambda$ is a convex polygon of which vertices are
exactly the vertices of $U_\lambda\cup L_\lambda$.

Note that any rhombus tiling turns into a combi without lenses by splitting
each $ij$-rhombus $\lozenge$ into two ``semi-rhombi'' $\Delta$ and $\nabla$ by
drawing the edge of type $ij$ connecting the left and right vertices.

Two properties of tilings are of a fundamental character:
 \begin{numitem1} \label{eq:s-coll_rt}
\cite{LZ} the map $T\mapsto V_T=:\Sscr$ gives a bijection between the set of
rhombus tilings $T$ on $Z(n,2)$ and the set of maximal s-collections $\Sscr$ in
$2^{[n]}$;
  \end{numitem1}
  \begin{numitem1} \label{eq:w-coll_combi}
\cite{DKK2} the map $K\mapsto V_K=:\Wscr$ gives a bijection between the set of
combies $K$ on $Z(n,2)$ and the set of maximal w-collections $\Wscr$ in
$2^{[n]}$.
  \end{numitem1}
(Note that in~\cite{LZ} the bijection exhibited in~\refeq{s-coll_rt} is given
in equivalent terms to rhombus tilings, namely, via the commutation classes of
pseudo-line arrangements.)
 \medskip

\noindent\textbf{Flips in set-systems and tilings.} Leclerc and
Zelevinsky~\cite{LZ} established two important facts on mutations, or flips, in
strongly and weakly separated collections:
  \begin{numitem1} \label{eq:set-flip}
Let $i<j<k$ be three elements of $[n]$ and let $X\subseteq[n]-\{i,j,k\}$.
 \begin{itemize}
\item[(i)]
If an s-collection $\Sscr\subseteq 2^{[n]}$ contains a set $U\in\{Xj,Xik\}$ and
the six sets (``witnesses'') $X$, $Xi$, $Xk$, $Xij$, $Xjk$, $Xijk$, then
replacing in $\Sscr$ the set $U$ by the other member of $\{Xj,Xik\}$, we again
obtain an s-collection.
\item[(ii)]
If a w-collection $\Wscr\subseteq 2^{[n]}$ contains a set $U\in\{Xj,Xik\}$ and
the four sets (``witnesses'') $Xi$, $Xk$, $Xij$, $Xjk$, then replacing in
$\Wscr$ the set $U$ by the other member of $\{Xj,Xik\}$, we again obtain a
w-collection.
  \end{itemize}
  \end{numitem1}

\noindent Following~\cite{LZ}, in case~(i), we call the replacement  $Xj$ by
$Xik$  (resp. $Xik$ by $Xj$) a \emph{raising} (resp. \emph{lowering})
\emph{flip} in $\Sscr$ ``in the presence of six witnesses'' (where the
adjectives are justified by the inequality $|Xj|<|Xik|$). A similar
transformation in case~(ii) is called a \emph{raising} (resp. \emph{lowering})
\emph{flip} in $\Wscr$ ``in the presence of four witnesses''. Let $\Gammas_n$
($\Gammaw_n$) be the directed graph whose vertex set is formed by the maximal
s-collections (resp. w-collections) in $2^{[n]}$ and whose edges are formed by
the pairs $(\Ascr,\Bscr)$ of collections where $\Bscr$ is obtained by one
raising flip from $\Ascr$. Since $\sum(|B|\colon B\in\Bscr)=\sum(|A|\colon
A\in\Ascr)+1$, both graphs $\Gammas_n$ and $\Gammaw_n$ are acyclic (have no
directed cycles). So they determine posets on the sets (classes) $\bfS_n$ and
$\bfW_n$, denoted as $\Poss_n$ and $\Posw_n$, respectively. These posets are
connected; moreover,
  \begin{numitem1} \label{eq:Gammasw}
both $\Poss_n$ and $\Posw_n$ have unique minimal and maximal elements, the
former consisting of all intervals, and the latter of all co-intervals in $[n]$
 \end{numitem1}
(where a \emph{co-interval} is the complement of an interval to $[n]$). Relying
on~\refeq{s-coll_rt} and~\refeq{w-coll_combi}, one can express the above flips
in a geometric form, as follows. Let $T$ be a rhombus or combined tiling on
$Z(n,2)$ and suppose that for some $i<j<k$ and $X\subseteq[n]-\{i,j,k\}$, $T$
contains five vertices $Xi,Xk,Xij,Xjk$ and $U\in \{Xj,Xik\}$.
Then~\refeq{s-coll_rt},\refeq{w-coll_combi},\refeq{set-flip} imply the
existence of a tiling $T'$ with the vertex set $V_{T'}$ equal to
$(V_T-\{U\})\cup\{U'\}$, where $\{U,U'\}=\{Xj,Xik\}$. It is known
(cf.~\cite[Prop.~3.2]{DKK2}) that if a tiling has vertices of the form $X'$ and
$X'i'$, then it has the $i'$-edge $(X',X'i')$. Therefore, assuming for
definiteness that $U=Xj$,
  \begin{itemize}
\item
$T$ contains the quadruple $Q$ of edges $(Xi,Xij), (Xk,Xjk), (Xj,Xij),
(Xj,Xjk)$, whereas $T'$ the quadruple $Q'$ of edges  $(Xi,Xij), (Xk,Xjk),
(Xi,Xik), (Xk,Xik)$.
 \end{itemize}
By a natural visualization, one says that the quadruple $Q$ forms an
\emph{M-configuration}, and $Q'$ a \emph{W-configuration}. Accordingly, the
transformation $T\leadsto T'$ is called a \emph{raising}, or \emph{M-to-W},
\emph{flip}, while $T'\leadsto T$  a \emph{lowering}, or \emph{W-to-M},
\emph{flip} on tilings. A simple fact is that $Q$ and $Q'$ are extendable to
larger structures, namely:
  \begin{numitem1} \label{eq:hexagon}
if a rhombus tiling $T$ contains a six-tuple of vertices as
in~\refeq{set-flip}(i), then $T$ contains six edges forming a hexagon on these
vertices, denoted as $H=H(X|ijk)$; furthermore, there is one extra vertex in
the interior of $H$, namely, either $Xj$ or $Xik$; and in the former (latter)
case, $H$ is subdivided into three rhombuses as illustrated in the left (resp.
right) fragment in Fig.~\ref{fig:hexagon}.
  \end{numitem1}

 \begin{figure}[htb]
\vspace{-0.3cm}
\begin{center}
\includegraphics[scale=1]{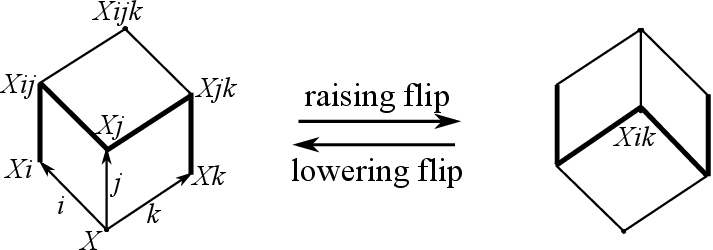}
\end{center}
\vspace{-0.3cm}
 \caption{Flips in rhombus tilings. M- and W-configurations are drawn in bold.}
 \label{fig:hexagon}
  \end{figure}
\vspace{-0cm}

So each M-to-W or  W-to-M flip within a hexagon in a rhombus tiling produces
another rhombus tiling. This gives rise to a poset on the rhombus tilings in
$Z(n,2)$, which is isomorphic to  $\Poss_n$ as above.

In case of a combi $K$, the flips involving $i,j,k,X$ as above, though changing
only one vertex of $K$, can change edges and tiles in a neighborhood of the
corresponding W- or M-configuration in a several possible ways; all of them are
described in~\cite[Sect.~3.3]{DKK2}. One possible way is illustrated in
Fig.~\ref{fig:flip-combi}.

   \begin{figure}[htb]
\vspace{-0cm}
\begin{center}
\includegraphics[scale=1]{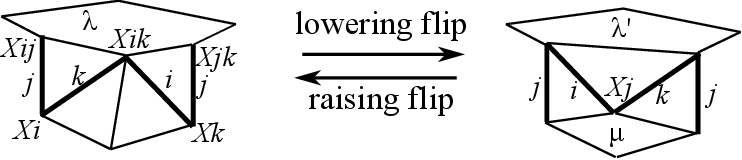}
\end{center}
\vspace{-0.3cm}
\caption{An example of lowering and raising flips in a combi.
Here $\lambda,\lambda', \mu$ are lenses.}
 \label{fig:flip-combi}
  \end{figure}
\vspace{-0cm}


\medskip
\noindent\textbf{Symmetric tilings and the middle line.} Define $m:=\lfloor
n/2\rfloor$; then $n=2m$ if $n$ is even, and $n=2m+1$ if $n$ is odd. An obvious
property of the $\ast$-symmetry is that
  \begin{numitem1} \label{eq:AAn}
any $A\subseteq [n]$ satisfies $|A|+|A^\ast|=n$.
  \end{numitem1}

Indeed, $\ast$ is the composition of two involutions on $2^{[n]}$, of which one
is defined  by $A\mapsto \bar A:=[n]-A$, and the other by $A\mapsto
A^\circ:=\{i^\circ\colon i\in A\}$. Then~\refeq{AAn} follows from the
equalities $|A|+|\bar A|=n$ and $|A|=|A^\circ|$.
  \smallskip

Next, in order to handle symmetric collections, it is convenient to assume that
\emph{the set $\Xi$ of generating vectors $\xi_i=(x_i,y_i)$ is symmetric}
w.r.t. the vertical line $VL:=\{(x,y)\in\Rset^2\colon x=0\}$ through the
origin, i.e.,
  \begin{equation} \label{eq:sym-gen}
  x_i=-x_{i^\circ} \quad\mbox{and}\quad y_i=y_{i^\circ},\quad i=1,\ldots,n.
  \end{equation}

This implies that the zonogon $Z:=Z(\Xi)$ is self-mirror-reflected w.r.t. $VL$.

The even and odd cases of $n$ differ in essential details and will be
considered separately. In the rest of this section we assume that $n$ is
\emph{even}, $n=2m$. Then
 \begin{numitem1} \label{eq:ML}
the zonogon $Z$ is self-mirror-reflected w.r.t. the horizontal line
 $$
  ML:=\{(x,y)\in\Rset^2\colon y=y_1+\cdots+y_{m}\},
$$
  \end{numitem1}
called the \emph{middle line}; we denote  $y_1+\cdots+y_{m}$ by $y^{ML}$.
From~\refeq{Xast} with $n$ even it follows that for $X\subseteq[n]$ and
$i\in[n]$, if $i,i^\circ\in X$ then $i,i^\circ\notin X^\ast$, and if $i\in
X\not\ni i^\circ$ then $i\in X^\ast\not\ni i^\circ$. This implies the following
nice property of the middle line $ML$:
  \begin{numitem1} \label{eq:mirrorA}
for any $A\subseteq[n]$, the sets $A$ and $A^\ast$ are \emph{mirror-reflected},
or \emph{symmetric}, to each other w.r.t. $ML$, i.e., their corresponding
points $(x_A,y_A)$ and $(x_{A^\ast},y_{A^\ast})$ in $\Rset^2$ satisfy
$x_A=x_{A^\ast}$ and $y_A-y^{ML}=y^{ML}-y_{A^\ast}$; in particular, a set
$A\subset[n]$ (regarded as a point in $\Rset^2$) lies on $ML$ if and only if
$A$ is self-symmetric: $A=A^\ast$.
 \end{numitem1}

\noindent\textbf{Definition.} A rhombus or combined tiling $T$ on $Z$ is
called \emph{symmetric} if the collection (vertex set) $V_T$ is symmetric.
 \medskip

Since a rhombus or combined tiling is determined by its vertex set (in view
of~\refeq{s-coll_rt},\refeq{w-coll_combi}), $T$ itself is symmetric, i.e., for
any tile of $T$, its mirror-reflected tile w.r.t. $ML$ belongs to $T$ as well.

One more useful property of $ML$ shown in~\cite{DKK4} (for $n$ even) is as follows:
  \begin{numitem1} \label{eq:combi-midline}
For a symmetric combi $K$ on $Z$, let $V_{K,ML}=(v_0,v_1,\ldots, v_r)$ be the
sequence of vertices of $K$ lying on $ML$ from left to right. Then $r=m$ and
there is a permutation $\sigma$ on $[m]$ such that for $i=1,\ldots, m$,
$v_{i}-v_{i-1}$ is congruent to the vector $\eps_{\sigma(i)\sigma(i^\circ)}$,
and the line segment $[v_{i-1},v_{i}]$ is either an edge of $K$ or the middle
section of some self-symmetric lens $\lambda$ in $K$ (i.e.,
$v_{i-1}=\ell_\lambda$ and $v_{i}=r_\lambda$). The vertices $v_0$ and $v_m$
represent the intervals $[m]$ and $[(m+1)..n]$, respectively.
  \end{numitem1}

We associate to $K$ the permutation $\sigma=\sigma_K$ on $[m]$ as
in~\refeq{combi-midline} and say that the combi $K$ is \emph{agreeable} with
$\sigma$, and similarly for $V_K$ and for $v_0,\ldots,v_m$.
  \medskip

\noindent\textbf{Remark 1.} ~We can split $K$ into the \emph{lower half-combi}
$\Klow$ and the \emph{upper half-combi} $\Kup$, the ``parts'' of $K$ lying in
the halves $\Zlow$ and $\Zup$ of $Z$ formed by the points $(x,y)$ with $y\le
y^{ML}$ and $y\ge y^{ML}$, respectively. Here if $ML$ cuts the interior of a
lens $\lambda$, then $\Klow$ acquires the \emph{lower semi-lens} $\lambdalow$,
and $\Kup$ the \emph{upper semi-lens} $\lambdaup$ (the parts of $\lambda$ below
and above $[\ell_\lambda,r_\lambda]$, respectively). Then $\Kup$ is symmetric
to $\Klow$.

Due to this, a symmetric combi is determined by its lower half-combi. More
precisely, fix a permutation $\sigma$ on $[m]$ and consider a half-combi $K'$
on $\Zlow$ consisting of $\Delta$- and $\nabla$-triangles, lenses and lower
semi-lenses at level $m$ so that the edges lying on $ML$ be of types $ii^\circ$
and follow in the order $\sigma$. Combining $K'$ and its symmetric half-combi
and removing redundant edges lying on $ML$, we obtain a correct symmetric
combi.
\medskip

Finally, since a rhombus tiling $T$ is in fact a particular case of combies (up
to cutting each rhombus $\lozenge$ into a pair of $\Delta$- and
$\nabla$-tiles), we can consider the \emph{lower} and \emph{upper
half-rhombus-tilings} $\Tlow$ and $\Tup$. Here each tile of $\Tlow$ ($\Tup$) is
either an entire rhombus or a $\nabla$-tile (resp. $\Delta$-tile) having one
edge on $ML$.


\section{Flips in symmetric s-collections}  \label{sec:scoll}

As mentioned above, the cases of $n$ even and odd differ essentially and are
considered separately. This section deals with symmetric strongly separated
collections in $2^{[n]}$ when the number $n$ of colors is even, $n=2m$. The
case of symmetric s-collections with $n$ odd will be considered in
Sect.~\SSEC{sw-n-odd} simultaneously with w-collections.

Consider a maximal symmetric s-collection $\Sscr$ in $2^{[n]}$ and the
corresponding rhombus tiling $T$ with $V_T=\Sscr$ (cf.~\refeq{s-coll_rt}).
By~\refeq{combi-midline}, there are exactly $m+1$ sets  $A_0,A_1,\ldots,A_m$ in
$\Sscr$ lying on the middle line $ML$ and ordered from left to right; they are
agreeable with a permutation $\sigma$ on $[m]$, namely, for $i=1,\ldots, m$,
~$A_i$ is obtained from $A_{i-1}$ by replacing the element $\sigma(i)$ by
$(\sigma(i))^\circ$.

Let us fix a permutation $\sigma$ and define $\bfSsym_{n}(\sigma)$ to be the
set of collections in $\bfSsym_n$ agreeable with $\sigma$. We first are going
to define appropriate flips within $\bfSsym_{n}(\sigma)$, and then will
introduce flips which connect certain representatives of $\bfSsym_{n}(\sigma)$
and $\bfSsym_{n}(\sigma')$ when permutations $\sigma$ and $\sigma'$ differ by
one transposition.

Flips in $\bfSsym_{n}(\sigma)$ are constructed in a natural way relying
on~\refeq{set-flip}(i). Here, acting in terms of rhombus tilings $T$ with
$V_T\in\bfSsym_{n}(\sigma)$, we choose a hexagon $H$ of $T$ lying in the lower
half $\Tlow$ and simultaneously make a flip within $H$ and its symmetric flip
in the symmetric hexagon $H^\ast$ lying in $\Tup$.

For a vertex $v$ of a tiling $T$, denote the set of (directed) edges entering
$v$ by $\deltain(v)=\deltain_T(v)$. Also let $|v|$ denote the level of $v$
(i.e., the number of elements when $v$ is regarded as a subset of $[n]$). We
rely on the next two lemmas.

\begin{lemma} \label{lm:deltain3}
Let $v$ be a vertex of $T$ such that {\rm(a)} $|\deltain(v)|\ge 3$, and
{\rm(b)} $|v|$ is minimum subject to~{\rm(a)}. Then $v$ is the top vertex of a
hexagon $H(X\,|\,i<j<k)$ having the W-configuration (i.e., formed by the
rhombuses $\lozenge(X| ik)$, $\lozenge(Xi| jk)$, $\lozenge(Xk| ij)$).
\end{lemma}
  \begin{proof}
(The method of proof is, in fact, well-known in the literature.) Let $e,e',e''$
be three consecutive edges (from left to right) in $\deltain(v)$ and let they
have colors $k>j>i$, respectively. Then $v$ is the top vertex of the rhombus of
$T$ containing $e,e'$, say, $\lozenge'=\lozenge(X'|jk)$, and the rhombus
containing $e',e''$, say, $\lozenge''=\lozenge(X''|ij)$. (So $v=X'jk=X''ij$.)
The edges $q'=(X',X'k)$ and $q''=(X'',X''i)$ are different and both enter the
beginning vertex $v'$ of the edge $e'$. By the minimality of $v$, we have
$|\deltain_T(v')|<3$; therefore, the edges $q'$ and $q''$ belong to the same
rhombus $\lozenge=\lozenge(X|ik)$ (where $X=X'-i=X''-k$). Combining
$\lozenge,\lozenge',\lozenge''$, we obtain the desired hexagon with the
W-configuration.
  \end{proof}

It follows that if we apply to a symmetric tiling agreeable with $\sigma$ a
sequence of lowering flips involving hexagons $H$ with the top vertex at level
$\le m$ (and simultaneously make raising flips in the corresponding symmetric
hexagons above the middle line) as long as possible, we eventually obtain a
tiling $T$ with $V_{T}\in\bfSsym_{n}(\sigma)$ such that
 \begin{numitem1} \label{eq:degin2}
any vertex $v$ of $T$ with $|v|\le m$ satisfies $|\deltain_T(v)|\le 2$.
  \end{numitem1}

\noindent\textbf{Definition.} A flip which makes a W-transformation below $ML$
and its  symmetric M-transformation above $ML$, is called a \emph{double
(hexagonal) flip agreeable with $\sigma$}.

 \begin{lemma} \label{lm:unique_2}
For each permutation $\sigma$ on $[n]$, there exists exactly one tiling $T$
with $V_T\in \bfSsym_{n}(\sigma)$ for which~\refeq{degin2} is valid.
 \end{lemma}
 \begin{proof}
First of all we check that $\bfSsym_{n}(\sigma)$ is nonempty. To see this, let
$A_0,A_1,\ldots,A_m$  be the sequence of points (subsets of $[n])$ on $ML$
agreeable with $\sigma$. Then $A_0=[m]$, and for each $i=1,\ldots,m$,
   $$
   A_i=([m]-\{\sigma(1),\ldots,\sigma(i)\})\cup\{(\sigma(1))^\circ,\ldots,(\sigma(i))^\circ\}.
   $$

This implies $(A_i-A_j)<(A_j-A_i)$ for any $i<j$ (in view of $a<b^\circ$ for
any $a,b\in[m]$). Therefore, the collection $\Ascr:=\{A_0,\ldots,A_m\}$ is
strongly separated. Moreover, one can check that the collection $\Ascr'$
obtained by adding to $\Ascr$ the sets $A_{i-1}\cup A_i$ and $A_{i-1}\cap A_i$
for all $i\in[m]$ is strongly separated as well.

Let $\lozenge_i$ be the $\sigma(i)\sigma(i^\circ)$-rhombus with the left vertex
$A_{i-1}$ and right vertex $A_i$, $i=1,\ldots,m$. The union of these rhombuses
has the vertex set just $\Ascr'$. Since $\Ascr'$ is strongly separated, there
exists a rhombus tiling $T$ with $V_T\supset\Ascr'$. This $T$ must contain the
rhombuses $\lozenge_1,\ldots,\lozenge_m$. Replacing in $T$ the half-tiling
$\Tup$ by the one symmetric to $\Tlow$, we obtain a symmetric tiling on
$Z(n,2)$ agreeable with $\sigma$, as required.

Since $\bfSsym_{n}(\sigma)$ is nonempty, it contains a tiling $T$
obeying~\refeq{degin2}. The uniqueness of such a $T$ follows from the
observation that for each $i=m,m-1,\ldots,1$, the set $V^i$ of vertices $v$ at
level $i$ in $T$ determines the next set $V^{i-1}$. (Indeed, if
$\lozenge,\lozenge'$ are two consecutive rhombuses in which the right vertex
$v$ of the former coincides with the left vertex of the latter and lies at
level $i$, then~\refeq{degin2} implies that the bottoms of
$\lozenge,\lozenge'$ form two consecutive vertices at level $i-1$.)
 \end{proof}

Denoting the tiling as in this lemma by $\Tmin_{n}(\sigma)$, we obtain the following

 \begin{corollary} \label{cor:Gsym_nsigma}
Let $\Gammasym_{n}(\sigma)$ be the directed graph with the vertex set
$\bfSsym_{n}(\sigma)$ whose edges are the pairs $(T,T')$ where the tiling $T$
is obtained by a double hexagonal flip from $T'$ which is lowering below $ML$.
Then $\Gammasym_{n}(\sigma)$ is connected and $\Tmin_{n}(\sigma)$ is its unique
minimal (zero-indegree) vertex.
  \end{corollary}

Hence $\Gammasym_{n}(\sigma)$ determines a poset, denoted as
$\Possym_{n}(\sigma)$, in which $\Tmin_{n}(\sigma)$ is the unique minimal
element. Figure~\ref{fig:n=6} illustrates all possible lower half-tilings for
$n=6$ and $\sigma=123,\,132,\,312$. Here the ones appeared from minimal tilings
$\Tmin_{n}(\sigma)$ are labeled $A$ (for $\sigma=123$), $B$ (for $\sigma=132$),
and $C$ (for $\sigma=312$).

 \begin{figure}[h]
\vspace{0cm}
\begin{center}
\includegraphics[scale=0.9]{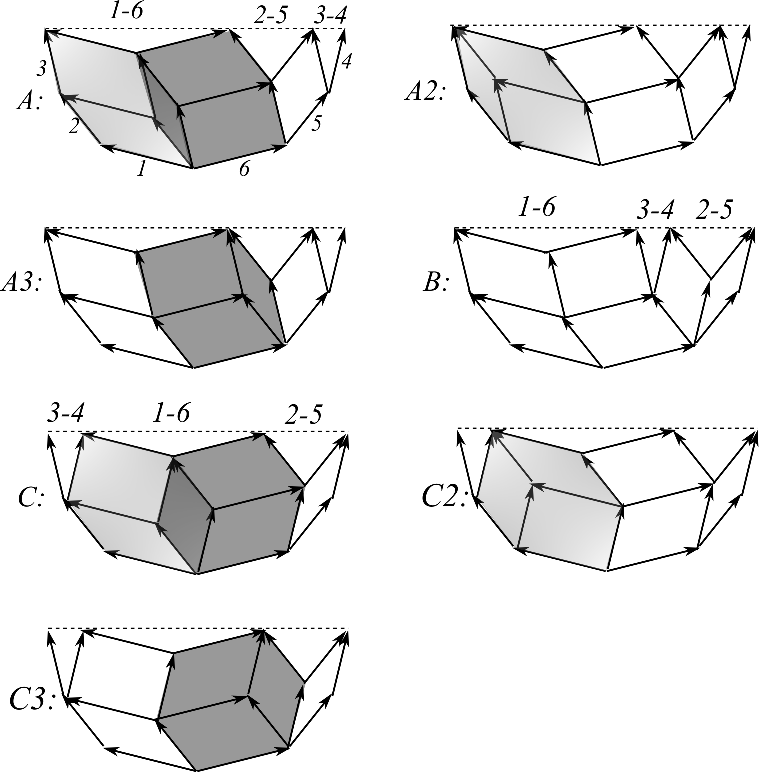}
\end{center}
\vspace{-0.2cm} \caption{Case $n=6$. The configurations (lower half-tilings)
for $\sigma=123$ (namely, A,A2,A3), $\sigma=132$  (namely, B), and $\sigma=312$
(namely, C,C2,C3).}
 \label{fig:n=6}
  \end{figure}

The graphs $\Gammasym_{n}(\sigma)$ for different $\sigma$'s are disjoint, and
in order to connect them we need to introduce one more sort of flips. Consider
two permutations $\sigma,\sigma'$ on $[m]$ that differ by one transposition:
there is $i$ such that $\sigma(i)=\sigma'(i+1)$, $\sigma(i+1)=\sigma'(i)$, and
$\sigma(j)=\sigma'(j)$ for $j\ne i,i+1$. Let for definiteness
$\sigma(i)<\sigma(i+1)$.

Let $A_0,A_1,\ldots,A_m$ be, as before, the corresponding sequence of vertices
(subsets of $[n]$) in the middle line for tilings agreeable with $\sigma$, and
$A'_0,A'_1,\ldots, A'_m$ a similar sequence for $\sigma'$. Then $A_i\ne A'_i$
and $A_j=A'_j$ for all $j\ne i$. For brevity we further write $\alpha$ for
$\sigma(i)$, and $\beta$ for $\sigma(i+1)$.

Form the auxiliary ``zonogon'' $\Omega$ with the origin at $A_{i-1}\cap
A_{i+1}$ and the generating vectors $\xi_\alpha,
\xi_\beta,\xi_{\beta^\circ},\xi_{\alpha^\circ}$. This $\Omega$ is
self-symmetric w.r.t. $ML$ and contains $A_{i-1},A_{i+1}$ on the boundary and
$A_i,A'_i$ in the interior. There are two self-symmetric tilings
$\Pi_\sigma,\,\Pi_{\sigma'}$ on $\Omega$ formed by six rhombuses (of types
$\alpha\alpha^\circ,\, \beta\beta^\circ,\, \alpha\beta,\, \alpha\beta^\circ,\,
\beta\alpha^\circ,\, \beta^\circ\alpha^\circ$), where $\Pi_\sigma$ contains the
vertex $A_i$, and $\Pi_{\sigma'}$ the vertex $A'_i$; they are illustrated in
Fig.~\ref{fig:n=4}.

\begin{figure}[h]
\vspace{-0cm}
\begin{center}
\includegraphics[scale=0.8]{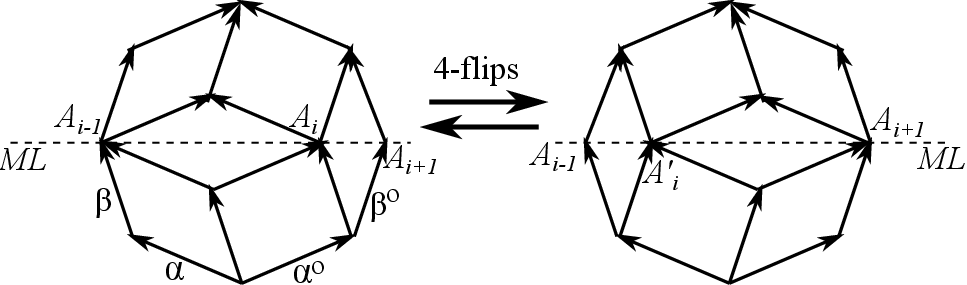}
\end{center}
\vspace{-0.3cm}
 \caption{big flips swapping fragments $\Pi_\sigma$ and $\Pi_{\sigma'}$}
 \label{fig:n=4}
  \end{figure}
\vspace{0cm}

The collection consisting of rhombuses $\lozenge_1,\ldots,\lozenge_{i-1},
\lozenge_{i+2},\ldots,\lozenge_m$ as in the proof of Lemma~\ref{lm:unique_2}
plus the ones in $\Pi_\sigma$ can be extended to a tiling $T$ in
$\bfSsym_{n}(\sigma)$ (this is shown using the fact that the lower boundary of
this collection forms a path from $[m]$ to $[(m+1)..n]$ consisting of $n$ edges
of different types). Replacing in $T$ the fragment $\Pi_\sigma$ by
$\Pi_{\sigma'}$, we obtain a tiling $T'$ in $\bfSsym_{n}(\sigma')$, and vice
versa.
 \medskip

\noindent\textbf{Definition.} The transformation $T\leadsto T'$ or $T'\leadsto
T$ as above is called a \emph{big (\emph{or} barrel) flip} turning the shape
$\Pi_\sigma$ into $\Pi_{\sigma'}$, or conversely.
  \medskip

\noindent(A transformation of this sort appeared in~\cite[Sect.~6]{eln} under the name of an \emph{octagon-flip}.) Now combine the above directed graphs $\Gammasym_{n}(\sigma)$ (over all
permutations $\sigma$) into one mixed graph $\Gammasym_n$ by adding undirected
edges to connect each pair of tilings of which one is obtained by a big flip
from the other. Figure~\ref{fig:flgraph-n=6} illustrates the graph
$\Gammasym_6$. Here $\bfSsym_6$ consists of 14 symmetric tillings, seven of
them are illustrated in Fig.~\ref{fig:n=6}, and the other ones (labeled with
primes) are their dual (agreeable with the permutations $321,\, 231,\, 213$).

\begin{figure}[b]
\vspace{0cm}
\begin{center}
\includegraphics[scale=0.8]{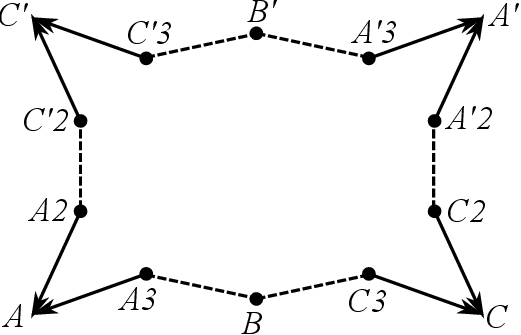}
\end{center}
\vspace{-0.3cm}
 \caption{The flip graph for $n=6$.}
 \label{fig:flgraph-n=6}
  \end{figure}
\vspace{-0cm}

Summing up the above constructions and reasonings, we obtain the following result; in terms of flips in symmetric tilings it was shown in~\cite[Sect.~6]{eln} (based on a general fact on reduced words for finite Coxeter groups). 

 \begin{theorem} \label{tm:all-s-flips}
Any two maximal symmetric s-collections in $2^{[n]}$ with $n$ even can be
linked by a series of flips of two sorts: double hexagonal flips (acting within
the same ``block'' $\bfSsym_{n}(\sigma)$) and big flips (linking collections
agreeable with neighboring $\sigma,\sigma'$). Therefore, $\Gammasym_n$ is
connected.
  \end{theorem}


\section{Flips in symmetric w-collections}  \label{sec:wcoll}

The first part of this section considers maximal symmetric weakly separated
collections in $2^{[n]}$ when $n$ is even, while the second one deals with the
case of $n$ odd.


\subsection{Maximal symmetric w-collections in $2^{[n]}$ with $n$ even.} \label{ssec:w-n-even}

Consider such a collection $\Wscr$ and the symmetric combi $K$ with $V_K=\Wscr$
on the zonogon $Z=Z(n,2)$ (existing by reasonings in Remark~1). Since any combi
without lenses is equivalent to a rhombus tiling, in which case symmetric flips
are described in the previous section, we may assume that the set of lenses in
$K$ is nonempty.

The symmetry of $K$ w.r.t. the middle line $ML$ of $Z$ (subject to~\refeq{ML})
implies that $K$ has at least one lens $\lambda$ whose lower boundary
$L_\lambda$ is contained in the lower half-combi $\Klow$. (In this case, either
$\lambda$ lies entirely in $\Klow$, or $\lambda$ is self-symmetric and its
vertices $\ell_\lambda$ and $r_\lambda$ lie on $ML$.)  It is known that the
directed graph whose vertices are the lenses of a combi and whose edges are the
pairs $(\lambda,\lambda')$ such that the upper boundary $U_\lambda$ and the
lower boundary $L_{\lambda'}$ share an edge is acyclic
(see~\cite[Sec.~3]{DKK2}). This implies that
  \begin{numitem1} \label{eq:min-lens}
for an arbitrary combi $K'$ on $Z(n',2)$ and $h\in[n']$, if $K'$ has at least
one lens of level $h$, then there is a lens $\lambda$ of this level such that
each edge in $L_\lambda$ is shared with a $\nabla$-tile, not with another lens
(where the level of a lens is the size $|A|$ of any of its vertices (regarded
as subsets of $[n']$)).
  \end{numitem1}

Returning to a symmetric combi $K$  on $Z=Z(n,2)$, take a lens $\lambda$ of
level $\le m$ in $K$ satisfying~\refeq{min-lens}  (where $n=2m$). Choose two
consecutive edges $e=(A,B)$ and $e'=(B,C)$ in $L_\lambda$, and let
$\nabla=\nabla(D|AB)$ and $\nabla'=\nabla(D'|BC)$ be the $\nabla$-tiles of $K$
containing these edges. By the definition of a lens, there are colors $k>j>i$
such that $Ak=Bj=Ci$ (which is the upper root of $\lambda$). Then the edge $e$
has type $jk$, and $e'$ type $ij$. It follows that the edges
$(D,A),\,(D,B),\,(D',B),\,(D',C)$ have colors $j,k,i,j$, respectively, and
therefore they form a W-configuration in $K$.

This leads to the corresponding symmetric flip in $K$ and $\Wscr$. In terms of
$\Wscr$, denoting $D\cap D'$ by $X$, one can see that $D=Xi$, $D'=Xk$, $A=Xij$,
$B=Xik$, $C=Xjk$. Then the \emph{symmetric weak flip} determined by
$\lambda,e,e'$ consists of the usual lowering (weak) flip $Xik\leadsto Xj$ (in
the ``presence of   four witnesses'' $Xi,\,Xk,\,Xij,\,Xjk$) and simultaneously
of the raising (weak) flip $(Xik)^\ast\leadsto (Xj)^\ast$. This combined flip
transforms $\Wscr$ into a symmetric collection $\Wscr'$.

In terms of $K$, the raising flip handles the lens $\lambda^\ast$ symmetric to
$\lambda$, the edges $e^\ast$ and $e'^\ast$ in $U_{\lambda^\ast}$, and the
$\Delta$-tiles $\Delta=\Delta(D^\ast|A^\ast B^\ast)$ and
$\Delta'=\Delta(D'^\ast|B^\ast C^\ast)$ whose edges $(A^\ast,D^\ast),
(B^\ast,D^\ast), (B^\ast,D'^\ast), (C^\ast,D'^\ast)$ form an M-configuration.
The lens $\lambda^\ast$ coincides with $\lambda$ if the level $h$ of $\lambda$
equals $m$, while $\lambda^\ast$ lies in $\Kup$ if $h<m$. As is seen from the
description of flips in~\cite[Sec.~3]{DKK2}, the former flip makes a local
transformation within $\Klow$ (a particular case is illustrated in
Fig.~\ref{fig:flip-combi}), while the latter flip does so within $\Kup$. These
flips act independently and result in a correct symmetric combi $K'$ with
$V_{K'}=\Wscr'$. Hence $\Wscr'$ is weakly separated.

The above double flip preserves the vertex structure on $ML$ and decreases by 1
the total size of vertices of the current combi that are contained in $\Zlow$.
If the new combi $K'$ still has a lens, one can repeat the procedure, and so
on. The process terminates when all lenses in the current combi vanish. Then we
can conclude with the following

 \begin{prop} \label{pr:wflips}
Let $\Wscr$ be a maximal w-collection agreeable with a permutation $\sigma$ on
$[m]$. Then one can apply to $\Wscr$ a series of symmetric double (weak) flips
as described above, preserving $\sigma$ during the process, so as to eventually
obtain an s-collection in $\bfSsym_{n}(\sigma)$.
  \end{prop}

Combining this with Theorem~\ref{tm:all-s-flips}, we come to the following

 \begin{corollary} \label{cor:all-w-flips}
Any two maximal symmetric w-collections in $2^{[n]}$ with $n$ even can be
linked by a series of flips of two sorts: double hexagonal or weak flips
(keeping permutations on $[m]$) and big flips (linking s-collections agreeable
with neighboring $\sigma,\sigma'$).
  \end{corollary}


\subsection{When the number $n$ of colors is odd.} \label{ssec:sw-n-odd}

Let $n$ be odd, $n=2m+1$. We first consider a \emph{maximal symmetric}
w-collection $\Cscr$ in $2^{[n]}$ and introduce symmetric flips for it, whereas
flips for s-collections will be obtained as a by-product in the end of this
section. It should be noted that the size of such a $\Cscr$ is strictly less
than that of a maximal \emph{non-symmetric} collection in $2^{[n]}$, i.e.,
$|\Cscr|<s_{n,1}$ (where $s_{n,r}$ is defined in Sect.~\SEC{intr});
see~\cite{DKK4}. We rely on the description of flips in the even case, with
$2m$ colors, in Sect.~\SEC{scoll}, and on a relationship between even and odd
cases established in~\cite[Sec.~4]{DKK4}.

Consider a maximal symmetric w-collection $\Cscr\in\bfWsym_n$. The color $m+1$
is self-symmetric: $(m+1)^\circ=m+1$, and from definition~\refeq{Xast} it
follows that for any $X\subseteq[n]$, exactly one of $X,X^\ast$ contains the
element $m+1$. In particular, $\Cscr$ has no self-symmetric sets. Partition
$\Cscr$ as $\Cscr'\sqcup \Cscr''$, where
  \begin{equation} \label{eq:CpCpp}
  \Cscr':=\{X\in\Cscr\colon m+1\notin X\}\quad\mbox{and}\quad
        \Cscr'':=\{X\in\Cscr\colon m+1\in X\}.
        \end{equation}

Using the fact that $\Cscr$ is weakly separated, one can show that
  \begin{equation} \label{eq:mm+1}
  |X|\le m\;\; \mbox{for all}\;\; X\in\Cscr'\quad \mbox{and}\quad |X|\ge m+1\;\; \mbox{for all}\;\; X\in\Cscr''
  \end{equation}
(cf.~\cite[Exp.~(4.1)]{DKK4}). Define
  \begin{equation} \label{eq:WpWpp}
  \Wscr':=\Cscr', \quad \Wscr'':=\{X-\{m+1\}\colon X\in\Cscr''\},
  \quad\mbox{and} \quad \Wscr:=\Wscr'\cup\Wscr''.
   \end{equation}

In what follows we denote the $2m$-element set $[n]-\{m+1\}$ as $[n]^-$,
keeping the definition of complementary colors: for $i\in[n]^-$,
$i^\circ=n+1-i$, and keeping the definition of set symmetry~\refeq{Xast}. In
view of~\refeq{mm+1}, the symmetric collections in $2^{[n]^-}$ are, in fact,
equivalent to those in $2^{[2m]}$, and instead of the set $\bfWsym_{2m}$ of
maximal symmetric w-collections in $2^{[2m]}$, we may deal with the equivalent
set for $2^{[n]^-}$, denoted as $\bfWsym([n]^-)$. The following relation (shown
in~\cite[Sec.~3]{DKK4}) is valid:
  \begin{numitem1} \label{eq:C-W}
the map $\Cscr\stackrel{\omega}\longmapsto\Wscr$ (where $\Wscr$ is obtained
from $\Cscr$ by~\refeq{CpCpp},\refeq{WpWpp}) gives a bijection between
$\bfWsym_n$ and $\bfWsym([n]^-)$.
  \end{numitem1}

Using this, we assign flips for $\bfWsym_{n}$ as natural analogs of flips for
$\bfWsym_{2m}$. More precisely, we associate to each $\Cscr\in\bfWsym_{n}$ the
corresponding permutation $\sigma$ on $[m]$, i.e., the one associated to
$\Wscr=\omega(\Cscr)$ according to~\refeq{combi-midline} (where $K$ is the
combi with $V_K=\Wscr$ and where the last vertex $v_m$ is now of the form
$[(m+2)..n]$). In other words, the sequence $\Mscr:=([m]=A_0,A_1,\ldots,
A_m=[(m+2)..n])$ of subsets (vertices) on the middle line $ML$ of the zonogon
$Z(2m,2)$ generated by the vectors $\xi_1,\ldots,\xi_m,\xi_{m+2},\ldots,\xi_n$
determines two sequences in $\Cscr$, both ``agreeable with $\sigma$''; namely,
$\Mscr$ and its symmetric (in $Z(n,2)$) sequence
  $$
  \Mscr^+:=(A_0^\ast=A_0\cup\{m+1\},\, A_1^\ast=A_1\cup\{m+1\},
  \ldots, A_m^\ast=A_m\cup\{m+1\})
  $$

Geometrically, we cut the combi $K$ along the middle line $ML$, preserve the
lower half-combi $\Klow$ (bounded from above by $ML$ containing $\Mscr$), and
shift the upper half-combi $\Kup$ in the vertical direction by the vector
$\xi_{m+1}$  (which is of the form $(0,y_{m+1})$, by~\refeq{sym-gen}). The
shifted half-combi, denoted as $\Kuppl$, is bounded from below by the line
$ML^+:=ML+\xi_{m+1}$ containing $\Mscr^+$ and has the vertex set
$\Cscr''=\{X\cup\{m+1\}\colon X\in V_{\Kup}=\Wscr''\}$. Connecting each vertex
$A_i$ with $A^\ast_i$ by the vertical segment congruent to $\xi_{m+1}$, we
obtain the polyhedral subdivision of $Z(n,2)$ formed by $\Klow$, $\Kuppl$ and
$m$ rectangles $\rho_1,\ldots,\rho_m$ between $ML$ and $ML^+$; we denote it by
$K^+$ and call a \emph{pseudo-combi}. An example with $n=7$ and $m=3$ is
illustrated in the picture.

\vspace{-0.1cm}
\begin{center}
\includegraphics[scale=0.9]{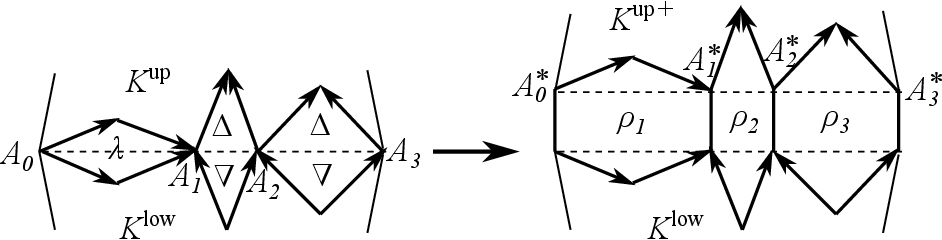}
\end{center}
\vspace{-0.2cm}

Based on the above observations, we have a natural correspondence between flips
in $\Wscr$ and $\Cscr$. Under this correspondence, every symmetric flip in
$\Wscr$, which replaces a set $Xik\in\Wscr'$  by $Xj$ (in the presence of
``four witnesses'' $Xi,Xk,Xij,Xjk$, where  $i<j<k$ and
$X\subseteq[m]-\{i,j,k\}$) and simultaneously replaces the set $Y\in \Wscr''$
symmetric to $Xik$ by $Y'$ symmetric to $Xj$, gives rise to an analogous flip
in $\Cscr$. The latter acts similarly in $\Cscr'=\Wscr'$ and replaces
$Y\cup\{m+1\}$ by $Y'\cup\{m+1\}$ in $\Cscr''$.
 \smallskip

Finally, let $\Cscr$ be \emph{strongly separated}. One easily checks that so is
$\Wscr$. Symmetric hexagonal flips in $\Wscr$ determine flips in $\Cscr$ as
described above.

As to a big flip in a strongly separated collection $\Wscr$, which replaces a
vertex $A_i$ by $A'_i$ (as shown in Fig.~\ref{fig:n=4}) and accordingly changes
the permutation $\sigma$ on $[m]$, the corresponding flip in $\Cscr$ replaces
the pair $A_i,A_i^\ast$ by $A'_i,{A'}^\ast_i$.


\section{Strongly $r$-separated collections and cubillages} \label{sec:rsep-cubil}

As a generalization of the usual strong separation, \cite{GP} introduced a
concept of (strong) $r$-separation for subsets of $[n]$. Recall that for
$r\in\Zset_{> 0}$, sets $A,B\subseteq[n]$ are called \emph{strongly
$r$-separated} if there are no elements $i_0<i_1<\ldots<i_{r+1}$ of $[n]$ that
alternate in $A-B$ and $B-A$. Accordingly, a collection $\Ascr\subseteq
2^{[n]}$ is called strongly $r$-separated if any two of its members are such.

In particular, the strong separated sets are strong 1-separated, and the chord
separated ones are 2-separated.  In what follows (except for Appendix~C) we
will deal with merely strong, not weak, $r$-separation, omitting the adjective
``strong'' everywhere.

In Sect.~\SEC{symm-r-even} we will consider symmetric $r$-separated collections
in $2^{[n]}$ when $n$ and $r$ are even, and  describe the flip structure for
the classes of \emph{size-maximal} collections $\Sscr$ among these, i.e., with
$|\Sscr|$ equal to $s_{n,r}$ defined in Sect.~\SEC{intr}. (The case with $n$
even and $r$ odd will be discussed in Appendix~\SEC{symm-r-odd}, while that
with $n$ odd and $r$ even in Appendix~\SEC{n-odd}.) These collections (in each
case) are \emph{representable}, which means that they can be represented by the
vertex sets of cubillages on a cyclic zonotope of dimension $r+1$. The purpose
of this section is to give necessary definitions and review some known results
on cubillages that will be used later. Some facts that we quote here can be
found in~\cite{DKK3}.

\subsection{Cyclic zonotope and cubillages.} \label{ssec:zon-cubil}

Let $n,d$ be integers with $n\ge d>1$. A \emph{cyclic configuration} of size
$n$ in $\Rset^d$ is meant to be an ordered set $\Xi$ of $n$ vectors
$\xi_i=(\xi_i(1),\ldots,\xi_i(d))\in\Rset^d$, $i=1,\ldots,n$, satisfying
  \begin{numitem1} \label{eq:cyc_conf}
\begin{itemize}
\item[(a)] $\xi_i(1)=1$ for each $i$, and
\item[(b)]
for the $d\times n$ matrix $A$ formed by $\xi_1,\ldots,\xi_n$ as columns
(in this order), any flag minor of $A$ is positive.
  \end{itemize}
  \end{numitem1}

(A typical sample of such a $\Xi$ is generated
by the Veronese curve; namely, take reals $t_1<t_2<\cdots<t_n$ and assign
$\xi_i:=\xi(t_i)$, where $\xi(t)=(1,t,t^2,\ldots,t^{d-1})$.)
  \medskip

\noindent\textbf{Definitions.}
The \emph{zonotope} $Z=Z(\Xi)$ generated by
$\Xi$ is the Minkowski sum of line segments $[0,\xi_i],\ldots,[0,\xi_n]$. A
\emph{cubillage} (called also a ``fine zonotopal tiling'' in the literature) is
a subdivision $Q$ of $Z$ into $d$-dimensional parallelotopes such that any two
either are disjoint or share a face, and each face of the boundary of $Z$ is
contained in some of these parallelotopes. For brevity, we refer to these
parallelotopes as \emph{cubes}.
 \medskip

When $n,d$ are fixed, the choice of one or another cyclic
configuration $\Xi$ (subject to~\refeq{cyc_conf}) does not matter in essence,
and we unify notation $Z(n,d)$ for $Z(\Xi)$, referring to it
as the \emph{cyclic zonotope} of dimension $d$ having $n$ colors.

Each subset $X\subseteq [n]$ naturally corresponds to the point $\sum_{i\in
X}\xi_i$ in $Z(n,d)$, and the cardinality $|X|$ is called the \emph{height} or
\emph{level} of this subset/point. (W.l.o.g., we usually assume that all
combinations of vectors $\xi_i$ with coefficients 0,1 are different.)

Depending on the context, we may think of a cubillage $Q$ on $Z(n,d)$ in two
ways: either as a set of $d$-dimensional cubes (and write $C\in Q$ for a cube
$C$) or as a polyhedral complex. The 0-, 1-, and $(d-1)$-dimensional faces of
$Q$ are called \emph{vertices}, \emph{edges}, and \emph{facets}, respectively.
By the subset-to-point correspondence, each vertex is identified with a subset
of $[n]$. In turn, each edge $e$ is a parallel translation of some segment
$[0,\xi_i]$; we say that $e$ has \emph{color} $i$, or is an $i$-\emph{edge}.
When needed, $e$ is regarded as a directed edge (according to the direction of
$\xi_i$). The set of vertices of $Q$ is denoted by $V_Q$. A face $F$ of $Q$ can
be denoted as $(X\,|\,T)$, where $X\subset[n]$ is the bottommost vertex, or
simply the \emph{bottom}, and $T\subset[n]$ is the set of colors of edges, or
the \emph{type}, of $F$ (note that $X\cap T=\emptyset$ always holds). An
important correspondence shown by Galashin and Postnikov is that
  \begin{numitem1} \label{eq:cub-dsepar}
\cite{GP} the map $Q\mapsto V_Q$ gives a bijection between the set of
cubillages on $Z(n,d)$ and the set of size-maximal $(d-1)$-separated
collections in $2^{[n]}$ (where the vertices are regarded as subsets of $[n]$).
   \end{numitem1}

In particular, every cubillage $Q$ can be uniquely restored from its vertex set
$V_Q$. An explicit construction is based on the following property (see,
e.g.,~\cite{DKK3}):
   \begin{numitem1} \label{eq:restore}
for $X,T\subset[n]$ with $X\cap T=\emptyset$, if a cubillage $Q$ contains the
vertices $X\cup S$ for all $S\subseteq T$, then $Q$ has the face $(X\,|\, T)$.
  \end{numitem1}

\subsection{Membranes and capsids.} \label{ssec:memb-capsid}

Certain subcomplexes in a cubillage are of importance to us. To define them,
let $\pi$ be the projection $\Rset^{d+1}\to\Rset^d$ given by
$x=(x(1),\ldots,x(d+1))\mapsto (x(1),\ldots,x(d))=:\pi(x)$ for
$x\in\Rset^{d+1}$. From~\refeq{cyc_conf} it follows that if a set $\Xi'$ of
vectors $\xi'_1,\ldots,\xi'_n$ forms a cyclic configuration in $\Rset^{d+1}$,
then the set $\Xi$ of their projections $\xi_i:=\pi(\xi'_i)$, $i=1,\ldots, n$,
is  a cyclic configuration in $\Rset^d$. So $Z(\Xi)=\pi (Z(\Xi'))$, and we may
liberally say that $\pi$ projects the zonotope $Z(n,d+1)$ onto $Z(n,d)$.

For a closed subset $U$ of points in $Z(n,d+1)$, let $U^{\rm fr}$ ($U^{\rm
rear}$) denote the subset of $U$ ``seen'' in the direction of the last,
$(d+1)$-th, coordinate vector $e_{d+1}$ (resp. $-e_{d+1}$), i.e., formed by the
points $x\in U$ such that there is no $y\in U$ with $\pi(y)=\pi(x)$ and
$y(d+1)<x(d+1)$ (resp. $y(d+1)>x(d+1)$). We call $U^{\rm fr}$ ($U^{\rm rear}$)
the \emph{front} (resp. \emph{rear}) \emph{side} of $U$.
\medskip

\noindent\textbf{Definition.} ~A \emph{membrane} of a cubillage $Q'$ on
$Z(n,d+1)$ is a subcomplex $M$ of $Q'$ such that $\pi$ homeomorphically
projects $M$ (regarded as a subset of $\Rset^{d+1}$) on $Z(n,d)$.
  \medskip

Then each facet of $Q'$ occurring in $M$ is projected to a cube of dimension
$d$ in $Z(n,d)$ and these cubes constitute a cubillage $Q$ on $Z(n,d)$, denoted
as $\pi(M)$ as well.

Sometimes it is useful to deal with a membrane $M$ in the zonotope
$Z'=Z(n,d+1)$ without specifying a cubillage on $Z'$ to which $M$ belongs. In
this case, $M$ is meant to be a $d$-dimensional polyhedral complex lying in
$Z'$ whose vertex set consists of subsets of $[n]$ (regarded as points) and
corresponds to the vertex set of some cubillage $Q$ on $Z=Z(n,d)$.
Equivalently, the projection $\pi$ establishes an isomorphism between $M$ and
$Q$. We call such an $M$ an \emph{(abstract) membrane} in $Z'$ and denote as
$M_Q$. Both notions of membranes are ``consistent'' since (see,
e.g.~\cite{DKK3})
  \begin{numitem1} \label{eq:membr-in-zon}
for any membrane $m$ in $Z'$, there exists a cubillage on $Z'$ containing $M$.
  \end{numitem1}

Two s-membranes in $Z'$ are of an especial interest. These are the \emph{front
side} $\Zpfr$ and the \emph{rear side} $\Zprear$ of $Z'$. Their projections
$\pi(\Zpfr)$ and $\pi(\Zprear)$ (regarded as complexes) are called the
\emph{standard} and \emph{anti-standard} cubillages on $Z(n,d)$, respectively.
Such cubillages in dimension $d=2$ (viz. rhombus tilings) with $n=4$ are drawn
in Fig.~\ref{fig:st-anst}.

\begin{figure}[htb]
 \vspace{-0.3cm}
\begin{center}
\includegraphics{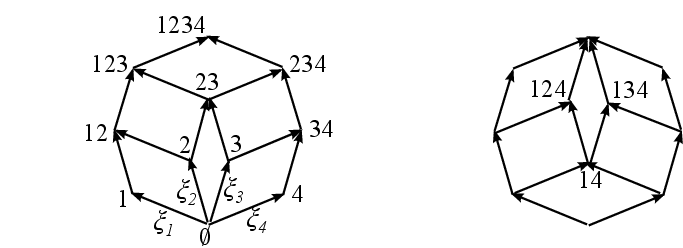}
\end{center}
\vspace{-0.5cm}
 \caption{left: standard tiling; right: anti-standard tiling}
 \label{fig:st-anst}
  \end{figure}

In particular, if $n=d+1$, then $Z'$ is nothing else than the
$(d+1)$-dimensional cube $(\emptyset\,|\,[d+1])$, and there are exactly two
membranes in $Z'$, namely, $\Zpfr$ and $\Zprear$.
  \medskip

\noindent\textbf{Definitions.}
Let $C=(X\,|\,T)$ be a $(d+1)$-dimensional cube
in $Z'=Z(n,d+1)$ whose type $T$ consists of elements $p_1<\cdots<p_{d+1}$ of
$[n]$. Following terminology in~\cite[Sec.~8]{DKK3}, the image $\pi(C)$ of $C$,
which is denoted as $(X\,|\,T)$ as well, is called the \emph{capsid} in
$Z=Z(n,d)$ with the bottom $X$ and type $T$. Since $\Cfr$ and $\Crear$ are the
only membranes in $C$, there are exactly two possible cubillages in the capsid
$\frakD=\pi(C)$, which are just $\pi(\Cfr)$ and $\pi(\Crear)$. The former
(latter) looks like the standard (resp. anti-standard) cubillages (using
terminology as above), and we say that $\frakD$ has the \emph{standard filling}
(resp. the \emph{anti-standard filling}), and denote it as $\frakDst$ (resp.
$\frakDant$).
 \medskip

The fillings of $\frakD$ are formed by the following cubes (cf.,
e.g.~\cite[Exp.~(A.1)]{DKK5}):
   \begin{numitem1} \label{eq:capsid}
   \begin{itemize}
\item[(i)]
$\frakDst$ consists of the cubes $F_i=(X\,|\,T-p_i)$ and
$G_j=(Xp_j\,|\,T-p_j)$, where $d-i$ is odd (i.e., $i=d+1,d-1,\ldots$) and $d-j$
is even ($j=d,d-2,\ldots$);
\item[(ii)]
$\frakDant$ consists of the cubes $F_i=(X\,|\,T-p_i)$ and
$G_j=(Xp_j\,|\,T-p_j)$, where $d-i$ is even and $d-j$ is odd.
  \end{itemize}
   \end{numitem1}

\noindent(As before, for disjoint subsets $A$ and $\{a,\ldots,b\}$ of $[n]$, we
use the abbreviated notation $Aa\ldots b$ for $A\cup\{a,\ldots,b\}$, and write
$A-c$ for $A-\{c\}$ when $c\in A$.)
 \medskip

Also one can check that (see~\cite[Prop.~8.1]{DKK3} or~\cite[Exp.~(3.2)]{DKK5})
  \begin{numitem1} \label{eq:inter_capsid}
for $\frakD$ as above, there is a unique vertex in the interior of $\frakDst$,
namely, $X\cup\{p_i\colon d-i\;\;{\rm even}\}$, denoted as $\Ist_{X,T}$, and a
unique vertex in the interior of $\frakDant$, namely,  $X\cup\{p_i\colon
d-i\;\; {\rm odd}\}$, denoted as $\Iant_{X,T}$.
  \end{numitem1}

\noindent\textbf{Definition.} Suppose that a cubillage $Q$ on $Z(n,d)$ contains
a capsid $\frakD$ as above having the standard (anti-standard) filling. Then
the replacement of $\frakDst$ by $\frakDant$ (resp. $\frakDant$ by $\frakDst$)
is called the \emph{raising flip} (resp. \emph{lowering flip}) in $Q$ using
$\frakD$. Such flips are denoted as $\frakDst\leadsto\frakDant$ and
$\frakDant\leadsto\frakDst$.
\medskip

\noindent(A similar mutation in a fine zonotopal tiling was introduced for
$d=3$ in~\cite[Sec.~3]{gal}). The resulting set of cubes is again a cubillage
on $Z$. Let $\bfQ_{n,d}$ denote the set of all cubillages on $Z(n,d)$. The
following property is of importance (cf.~\cite[Th.~D.1]{DKK3}):
  \begin{numitem1} \label{eq:cap_fl_poset}
for $n,d$ arbitrary, the directed graph $\Gamma_{n,d}$ whose vertex set is
$\bfQ_{n,d}$ and whose edges are the pairs $(Q,Q')$ such that $Q'$ is obtained
from $Q$ by a raising flip using some capsid is \emph{acyclic} and has unique
minimal (zero-indegree) and maximal (zero-outdegree) vertices, which are the
standard cubillage $\Qst_{n,d}$ and the anti-standard cubillage $\Qant_{n,d}$
on $Z(n,d)$, respectively.
\end{numitem1}

As a consequence, any two cubillages on $Z(n,d)$ can be connected by a series
of capsid flips, and $\Gamma_{n,d}$ determines a poset with the minimal element
$\Qst_{n,d}$ and the maximal element $\Qant_{n,d}$.


\subsection{Partial order on cubes and inversions.} \label{ssec:invers}

For two cubes $C,C'$ of a cubillage $Q$, if the rear side $\Crear$ of $C$ and
the front side $\Cpfr$ of $C'$  share a facet, we say that $C$
\emph{immediately precedes} $C'$. A known fact if that
  \begin{numitem1} \label{eq:nat_order}
the directed graph whose vertices are the cubes of a cubillage $Q$ and whose
edges are the pairs $(C,C')$ such that $C$ immediately precedes $C'$ is
acyclic.
  \end{numitem1}
This determines a partial order on the set of cubes of $Q$, called
in~\cite{DKK3} the \emph{natural order} on $Q$; we denote it as $(Q,\prec)$ or
$\prec_Q$.

We also will use the fact that the restriction of $\prec_Q$ to a capsid gives a
linear order. More precisely, the following is valid
(cf.~\cite[Prop.~10.1]{DKK3}):
  \begin{numitem1} \label{eq:order_capsid}
For a capsid $\frakD=(X\,|\,T=(p_1<p_2<\cdots<p_{d+1}))$ in $Q$ and for
$i=1,\ldots, d+1$, let $C_i$ denote the cube in the filling of $\frakD$ having
the type $T-p_i$ (which exists and unique by~\refeq{capsid}). Then
 \begin{itemize}
\item[(i)]
$C_{d+1}\prec C_d\prec \cdots\prec C_1$ if $\frakD$ has the standard filling
$\frakDst$, and
\item[(ii)]
$C_1\prec C_2\prec\cdots \prec C_{d+1}$  if $\frakD$ has the anti-standard
filling $\frakDant$.
  \end{itemize}
  \end{numitem1}

Next, the set $\Mscr(Q')$ of membranes of a cubillage $Q'$ on $Z'=Z(n,d+1)$
forms a distributive lattice. To see this, let us associate with a membrane
$M\in\Mscr(Q')$ the set of cubes of $Q'$, denoted as $Q'(M)$, lying between
$\Zpfr$ and $M$ (i.e., \emph{before} $M$, in a sense). One easily shows that if
$C,C'$ are cubes in $Q'$ such that $C$ immediately precedes $C'$ and $C'\in
Q'(M)$, then $C\in Q'(M)$ as well. This implies that $Q'(M)$ forms an ideal of
the natural order $\prec_{Q'}$. Conversely, any ideal of $\prec_{Q'}$ is
representable as $Q'(M)$ for some membrane $M$ of $Q'$. It follows that
     \begin{numitem1} \label{eq:s-lattice}
the set $\Mscr(Q')$ of membranes of a cubillage $Q'$ on $Z'=Z(n,d+1)$ is a
distributive lattice in which for $M,M'\in\Mscr(Q')$, the w-membranes $M\wedge
M'$ and $M\vee M'$ satisfy $Q'(M\wedge M')=Q'(M)\cap Q'(M')$ and
$Q'(M\vee M')=Q'(M)\cup Q'(M')$; the minimal and maximal elements
of this lattice are $\Zpfr$ and $\Zprear$, respectively.
  \end{numitem1}

Another important known fact is that the set of \emph{types} $T$ of the cubes
$C=(X\,|\,T)$ in the ideal $Q'(M)$ does not depend on $Q'$, in the sense that
any two cubillages on $Z'$ containing the same membrane $M$ have equal sets of
types of cubes before $M$ (see, e.g.,~\cite{KV,zieg}).  This set of types is
denoted as  $\Inver(M)$ and  called the set of \emph{inversions} of $M$. We
also say that this is the set of inversions of the cubillage $Q=\pi(M)$ on
$Z(n,d)$ and use notation $\Inver(Q)$ for $\Inver(M)$.


\subsection{Symmetric cubillages.} \label{ssec:sym_cub}

A cubillage $Q$ on $Z=Z(n,d)$ is called \emph{symmetric} if its vertex set
$V_Q$ is symmetric. By~\refeq{cub-dsepar}, such cubillages are bijective to
symmetric (strongly) $r$-separated collections $\Sscr$ with $r=d-1$ and
$|\Sscr|=s_{n,r}$ (earlier we have called such collections size-maximal and
representable and denoted their set by $\bfSsym_{n,d}$).

In fact, the $\ast$-symmetry on the vertices of $Q$ is extended in a natural
way to the faces (edges, facets, cubes) of $Q$. More precisely, if
$F=(X\,|\,T)$ is a face of $Q$ with $T=(i_1<\cdots<i_r)$, where $r\le d$, then
$F$ has the vertex set $V_F=\{XS=X\cup S\colon S\subseteq T\}$. By the symmetry
of $V_Q$, $Q$ contains the collection of vertices (subsets of $[n]$) symmetric
to those in $V_F$; this collection is viewed as $\{(XT)^\ast \cup S'\colon
S'\subseteq T^\circ\}$, where we write $T^\circ$ for
$(i_r^\circ<\cdots<i_1^\circ)$. (This follows from the identity
$(XS)^\ast=(XT)^\ast\cup (T^\circ-S^\circ)$ for any $S\subseteq T$, which is
valid when $X\cap T=\emptyset$; a verification of the identity is
straightforward and we leave it to the reader as an exercise.) Then,
by~\refeq{restore}, $Q$ contains the face $((XT)^\ast\,|\, T^\circ)$, which is
regarded as symmetric to $F$ and denoted as $F^\ast$. (When $|T|=d$, we obtain
a cube $C$ of $Q$ and its symmetric cube $C^\ast$.)

Strictly speaking, the above construction gives a symmetric
``combinatorial-cubic'' complex embedded in $Z(n,d)$. However, when needed, we
can use a ``purely geometric'' definition, as follows. Define the generating
vectors $\xi_1,\ldots \xi_n\in \Rset^d$  as in~\refeq{cyc_conf} in a
``symmetrized Veronese form'', by
  \begin{numitem1} \label{eq:symm_cyc_gen}
$\xi_i=(1,t_i,t_i^2,\ldots, t_i^{d-1})$; ~$t_1<\cdots<t_n$; ~and
$t_{i^\circ}=-t_i$, where $i^\circ:=n+1-i$.
  \end{numitem1}
(Also we assume that all 0,1-combinations of these vectors are different.) In
particular, for $i=1,\ldots,\lfloor n/2\rfloor$, we have $t_i<0$,
~$t_{i^\circ}>0$, and $j$-th coordinate of $\xi_{i^\circ}$ is $(-t_i)^{j-1}$.
The zonotope $Z=Z(\Xi=(\xi_1,\ldots,\xi_n))$ has the center at the point
$\zeta_Z:=\frac12(\xi_1+\cdots+\xi_n)$ and admits two involutions:
  \smallskip

(a) the central symmetry $\nu$ w.r.t. $\zeta_Z$, which sends $x$ to
$\nu(x)=2\zeta_Z-x$, and

(b) the symmetry $\mu$ w.r.t. the subspace $\{x\in \Rset^d\colon x(p)=0$ if $p$
is even$\}$, which sends $x\in Z$ to $\mu(x)=(x(1),-x(2),\ldots,
(-1)^{j-1}x(j),\ldots)$.
  \smallskip

The composition $\sigma:=\mu\circ\nu=\nu\circ\mu$ is again an involution on
$Z$, and one can check that the linear map $\sigma$ gives the desired symmetry
$X\mapsto X^\ast$ on the subsets $X\subseteq [n]$ (regarded as points
$\sum(\xi_i\colon i\in X)$). Moreover, $\sigma$ is orthonormal, and for each
cube $C=(X\,|\,T)$ in $Z$, $\sigma(C)$ is congruent (up to reversing) to $C$,
giving the cube $C^\ast$.

Using this geometric setting, let us demonstrate the following useful fact.

   \begin{lemma} \label{lm:fr-rear}
{\rm(i)} If $d$ is even, then the front side $\Zfr$ of $Z=Z(n,d)$ is
self-symmetric, and similarly for the rear side $\Zrear$.

{\rm(ii)} If $d$ is odd, then $\Zfr$ and $\Zrear$ are symmetric to each other.
  \end{lemma}
  \begin{proof}
For a point $x$ in $Z$, consider the points $u=\nu(x)$, $y=\mu(x)$, and
$z=\sigma(x)$. Since $\nu$ is the central symmetry on $Z$, $x$ lies on $\Zfr$
if and only if $u$ lies on $\Zrear$, and similarly, $y\in\Zfr$ if and only if
$z\in\Zrear$. Let $x\in \Zfr$. It suffices to show that
  \smallskip

($\ast$) ~if $d$ is even, then $y\in \Zrear$ (implying $z\in\Zfr$);

($\ast\ast$) ~if $d$ is odd, then $y\in\Zfr$ (implying $z\in\Zrear$).
  \smallskip

(Then ($\ast$) gives~(i) in the lemma (concerning $\Zfr$), and ($\ast\ast$)
gives (ii).)

To show ($\ast$) and ($\ast\ast$), represent $x$ as $\sum(\lambda_i\xi_i\colon
i\in X)$, where $X\subseteq[n]$ and $\lambda>0$. Then
$y=\sum(\lambda_i\xi_{i^\circ}\colon i\in X)$. Comparing the last coordinates
of $x$ and $y$, we observe that $y(d)=-x(d)$ if $d$ is even, and $y(d)=x(d)$ if
$d$ is odd (since $\xi_{i^\circ}(d)=t_{i^\circ}^{d-1}$,
$\xi_{i}(d)=t_{i}^{d-1}$, and $t_{i^\circ}=-t_i$).

For $d$ even, suppose that $y\notin \Zrear$. Then there is a point $y'$ in $Z$
such that $y'(p)=y(p)$ for $p=1,\ldots,d-1$, and $y'(d)<y(d)$. Taking the point
$x'=\mu(y')$, we obtain $x'(p)=x(p)$ for $p=1,\ldots,d-1$, and
$x'(d)=-y'(d)>-y(d)=x(d)$, contrary to the fact that $x\in\Zfr$.

When $d$ is odd, we argue in a similar way.
  \end{proof}

Considering the projection $\pi$ on $Z(n,d)$, we obtain from this lemma that
 \begin{numitem1} \label{eq:stand-antistand}
 \begin{itemize}
\item[(i)]
for $d'$ odd, each of the standard and anti-standard cubillages on $Z(n,d')$ is
self-symmetric;
\item[(ii)]
for $d'$ even, the standard and anti-standard cubillages on $Z(n,d')$ are
symmetric to each other.
  \end{itemize}
  \end{numitem1}


\section{Flips in symmetric strongly $r$-separated collections and symmetric
$(r+1)$-dimensional cubillages  when $r$ is even.}  \label{sec:symm-r-even}

This section deals with size-maximal (viz. representable) symmetric
$r$-separated collections in $2^{[n]}$ when both $n$ and $r$ are even. It turns
out that in this case the flip structure has a nice property: it forms a poset
with a unique minimal and a unique maximal elements, as we show in
Theorem~\ref{tm:Gammasym} and Corollary~\ref{cor:flip-rsepar}. Since any
maximal chord separated  (viz. strongly 2-separated) collection in $2^{[n]}$ is
representable, due to~\cite{gal}, we will obtain a nice characterization of the
flip structure for symmetric c-collections with $n$ even.

Let $\Sscr\subset 2^{[n]}$ be an $r$-separated collection which is size-maximal
and symmetric (such an $\Sscr$ exists by Corollary~\ref{cor:zon-sym-cub}
below). Let $d:=r+1$; then $d$ is odd. By~\refeq{cub-dsepar}, there exists a
cubillage $Q$ on $Z(n,d)$ such that $V_Q=\Sscr$; this $Q$ is symmetric.

First of all, using terminology and notation as in Sect.~\SSEC{memb-capsid},
consider a capsid $\frakD=(X\,|\,T)$ in $Q$; let $T=(p_1<\cdots<p_{d+1})$.
Taking the cubes symmetric to those occurring in $\frakD$, we obtain a capsid
in $Q$ as well, denoted as $\frakD^\ast$. We need two lemmas.

\begin{lemma} \label{lm:stan-stan}
If $\frakD$ has the standard filling, then so does $\frakD^\ast$, and similarly
when the filling of $\frakD$ is anti-standard.
  \end{lemma}
  \begin{proof}
Since $\frakD$ has the top vertex $XT$ and its edges are of colors
$p_1<\cdots<p_{d+1}$, $\frakD^\ast$ has the bottom $(XT)^\ast$ and its edges
have the colors $p^\circ_{d+1}<\cdots<p^\circ_1$. Hence
$\frakD^\ast=((XT)^\ast\,|\, T^\circ)$, where
$T^\circ=\{p^\circ_{d+1},\ldots,p^\circ_1\}$. Let us examine cubes in $\frakD$
and $\frakD^\ast$, using notation as in~\refeq{capsid}.

We say that a cube $C=(X'\,|\,T')$ in a filling of $\frakD$ is \emph{lower}
(\emph{upper}) if $X'=X$ (resp. $X'T'=XT$), i.e., it is of the form
$F_i=(X\,|\,T-p_i)$ for some $i$ (resp. of the form $G_j=(Xp_j\,|\, T-p_j)$ for
some $j$). When $d-i$ or $d-j$ is odd (even) we say that the cube $C$ is
\emph{odd} (resp. \emph{even}). And similarly for the capsid $\frakD^\ast$.

One can check that for a lower cube $F_i=(X\,|\,T-p_i)$ in $\frakD$, its
symmetric cube $F_i^\ast$ in $\frakD^\ast$ is viewed as $((XT)^\ast
p_i^\circ\,|\, T^\circ-p_i^\circ)$. Therefore, $F^\ast_i$ is upper in
$\frakD^\ast$. Since $p_i^\circ$ is $(d+2-i)$-th element in the ordered
$T^\circ$ and $d$ is odd, if $F_i$ is odd (even) in $\frakD$, then $F_i^\ast$
is even (resp. odd)  in $\frakD^\ast$. As to an upper cube
$G_j=(Xp_j\,|\,T-p_j)$ in $\frakD$, its symmetric cube $G^\ast_j$ in
$\frakD^\ast$ is viewed as $((XT)^\ast\,|\, T^\circ-p_j^\circ)=(X'\,|\,
T'-p'_{d+2-j})$, implying that $G_j^\ast$ is lower, and if $G_j$ is odd (even),
then $G_j^\ast$ is even (odd).

Now suppose that $\frakD$ has the standard filling $\frakDst$. Then,
by~\refeq{capsid}(i), each lower cube in it is odd, and each upper cube is
even. From the above analysis it follows that each upper cube in the filling of
$\frakD^\ast$ is even, while each lower cube is odd. This means that
$\frakD^\ast$ has the standard filling $\frakDastst$, as required.

By similar reasonings, when $\frakD$ has the anti-standard filling, so does
$\frakD^\ast$ as well.
  \end{proof}

\begin{lemma} \label{lm:C-Cast}
Either $\frakD=\frakD^\ast$, or $\frakD$ and $\frakD^\ast$ have no cube of $Q$
in common.
 \end{lemma}
 \begin{proof}
We use the following
 \medskip

\noindent\textbf{Claim}
~\emph{For arbitrary integers $n>d\ge 2$, let
$\frakD=(X\,|\,T)$ be a capsid in a cubillage $Q$ on $Z(n,d)$. Let $\Nscr$ be a
nonempty set of cubes contained in one of $\frakDst,\frakDant$ such that the
union $\Omega$ of these cubes is convex and different from the whole $\frakD$.
Then $\Nscr$ consists of exactly one cube.}
  \medskip

\noindent\textbf{Proof of the Claim.}
Let for definiteness $\Nscr$ is contained
in $\frakDst$ (when $\Nscr\subset\frakDant$, the argument is similar). Let
$\Nscr_1$ ($\Nscr_2$) be the set of lower (resp. upper) cubes in $\Nscr$, i.e.,
of the form $F_i$ (resp. $G_j$) in~\refeq{capsid}(i).

Suppose that $\Nscr_1$ contains two different cubes $C'=(X\,|\,T')$ and
$C''=(X\,|\,T'')$. Then $T'\cup T''=T$, and by the convexity of $\Omega$ and
$\frakD$, ~$\Nscr_1$ coincides with the set $\Dscr_1$ of all lower cubes in
$\frakDst$. Similarly, if $|\Nscr_2|\ge 2$, then $\Nscr_2$ is the set $\Dscr_2$
of upper cubes in $\frakDst$. Hence  at least one of $|\Nscr_1|,|\Nscr_2|$ is 0
or 1 (for otherwise $\Omega=\frakD$).

Note that any cube $F_i=(X\,|\,T-p_i)$ in $\Dscr_1$ and any cube
$G_j=(Xp_j\,|\,T-p_j)$ in $\Dscr_2$ share a facet in common, namely,
$(Xp_j\,|T-\{p_i,p_j\})$, where $T=(p_1<\cdots<p_{d+1})$. (In fact, any two
cubes in a capsid share a facet, but this is not needed to us.) This implies
that if $|\Nscr_k|\ge 2$ for some $k\in\{1,2\}$, then any cube $C$ in
$\Dscr_{3-k}$ has at least two facets shared with cubes in $\Nscr_k$, and by
the convexity of $\Omega$,  $C$ must belong to $\Nscr$.

It remains to consider the case $|\Nscr_1|=|\Nscr_2|=1$. Let $F_i$ be the cube
in $\Nscr_1$, and $G_j$ the cube in $\Nscr_2$. Let $R$ be their common facet
(of type $T-\{i,j\}$). Then the cube $F_i$ is the Minkowsky sum of $R$ and the
segment $[0,-\xi_i]$, while $G_j$ is the sum of $R$ and $[0,\xi_j]$. Since
$\Omega$ is convex, it contains the convex hull of  $F_i\cup G_j$, and by
evident geometric reasons, the latter is strictly larger than $F_i\cup G_j$
itself (taking into account that $\xi_i$ and $\xi_j$ non-colinear). This
contradiction implies $|\Nscr|=1$, as required.
 \hfill \qed
  \medskip

\noindent\textbf{Remark 2.} Using a method from~\cite[Sec.~4]{DKK3}), one can
obtain a sharper property, namely: if the union of some cubes of a cubillage on
a zonotope $Z$ forms a convex region $\Omega$, then $\Omega$ is representable
as a subzonotope in $Z$ (we omit a proof here).
  \medskip

Return to the proof of the lemma. Suppose that $\Omega:=\Dscr\cap \Dscr^\ast$
is different from $\Dscr$ and contains a cube of $Q$. Since $\Omega$ is convex,
it consists of exactly one cube $C=(X\,|\, T)$, by the Claim. Moreover, $C$ is
self-symmetric, $C=C^\ast$. But then the color set $T$ must be self-symmetric
as well. This is impossible since $|T|=d$ is odd and $n$ is even.
 \end{proof}

Based on the above lemmas, we define the desired flips in a symmetric cubillage
$Q$ on $Z(n,d)$ when $n$ is even and $d$ is odd as follows. Note that a capsid
$\frakD=(X\,|\, T)$ in $Q$ is self-symmetric (i.e., $\frakD=\frakD^\ast$) if
and only if $X=(XT)^\ast$ and $T=T^\circ$.
\medskip

\noindent\textbf{Definition.} ~For a capsid $\frakD$ in $Q$, the symmetric
raising (lowering) flip using $\frakD$ consists in the single flip
$\frakDst\leadsto \frakDant$ (resp. $\frakDant\leadsto \frakDst$) when $\frakD$
is self-symmetric, and the pair of raising (resp. lowering) flips, one
occurring in $\frakD$ and the other in $\frakD^\ast$, when
$\frakD\ne\frakD^\ast$. We also call such a transformation in $Q$ a
\emph{central flip} in the former case, and a \emph{double} flip in the latter
case.
 \medskip

Clearly such a flip makes again a symmetric cubillage on $Z(n,d)$. Moreover,
whenever $Q$ has a capsid with the standard (anti-standard) filling, we are
always able to make a symmetric raising (resp. lowering) flip  in $Q$. Let
$\bfQsym_{n,d}$ denote the set of symmetric cubillages on $Z(n,d)$. Relying
on~\refeq{cap_fl_poset}, we can conclude with the following

\begin{theorem} \label{tm:Gammasym}
For $n$ even and $d$ odd, the directed graph $\Gammasym_{n,d}$ whose vertex set
is $\bfQsym_{n,d}$ and whose edges are the pairs $(Q,Q')$ such that $Q'$ is
obtained from $Q$ by one symmetric (central or double) raising flip is acyclic
and has unique minimal (zero-indegree) and maximal (zero-outdegree) vertices,
which are the standard cubillage $\Qst_{n,d}$ and the anti-standard cubillage
$\Qant_{n,d}$ on $Z(n,d)$, respectively.
 \end{theorem}
(Note that by the minimality of $\Qst_{n,d}$, any capsid in it has the standard
filling. Similarly, by the maximality of $\Qant_{n,d}$, any capsid in it has
the anti-standard  filling. Each of these two cubillages is symmetric,
by~\refeq{stand-antistand}.)

As a consequence, we obtain a description of symmetric flips in the set
$\bfSsym_{n,r}$ of representable symmetric $r$-separated collections in
$2^{[n]}$ with $n,r$ even and $r\ge 2$. More precisely, for $\Sscr\in
\bfSsym_{n,r}$, consider the symmetric cubillage $Q$ on $Z(n,d=r+1)$ with
$V_Q=\Sscr$ and a capsid $\frakD=(X\,|\, T)$ in $Q$. The set $\{XS\colon
S\subseteq T\}$ of (boundary) vertices of $\frakD$ forms a subcollection in
$\Sscr$, denoted as $\Nscr_{X,T}$. If $\frakD$ has the standard (anti-standard)
filling, then, besides $\Nscr_{X,T}$, ~$\frakD$ contains one more member of
$\Sscr$, namely, the set corresponding to the vertex $\Ist_{X,T}$ (resp.
$\Iant_{X,T}$) occurring in the interior of $\frakD$, as indicated
in~\refeq{inter_capsid}. The above observations and results give rise to the
following definition and corollary from Theorem~\ref{tm:Gammasym}.
   \medskip

\noindent\textbf{Definition.} Suppose that $\Sscr\in\bfSsym_{n,r}$ includes the
subcollection $\Nscr_{X,T}$ for some $X,T\subset[n]$ with $X\cap T=\emptyset$
and $|T|=r+2$. Then (by the size-maximality) $\Sscr$ must contain the set
either $\Ist_{X,T}$ or $\Iant_{X,T}$. The symmetric raising (lowering) flip
w.r.t. $(X,T)$ consists of the single replacement of  $\Ist_{X,T}$ by
$\Iant_{X,T}$ (resp. $\Iant_{X,T}$ by $\Ist_{X,T}$) when $\Nscr_{X,T}$ is
self-symmetric, and the pair of symmetric replacements $\Ist_{X,T}\leadsto
\Iant_{X,T}$ and $\Ist_{(XT)^\ast,T^\circ}\leadsto \Iant_{(XT)^\ast,T^\circ}$
(resp. $\Iant_{X,T}\leadsto \Ist_{X,T}$ and $\Iant_{(XT)^\ast,T^\circ}\leadsto
\Ist_{(XT)^\ast,T^\circ}$) otherwise.

 \begin{corollary} \label{cor:flip-rsepar}
For $n,r$ even, the directed graph with the vertex set $\bfSsym_{n,r}$ whose
edges are the pairs $(\Sscr,\Sscr')$ such that $\Sscr'$ is obtained from
$\Sscr$ by a symmetric raising flip as in the above definition is acyclic and
determines a poset on $\bfSsym_{n,r}$ having unique minimal and maximal
elements, which are the vertex sets of the standard and anti-standard
cubillages on $Z(n,r+1)$, respectively.
 \end{corollary}

It is known that for $r$ even, the vertex set of $\Qst_{n,r+1}$  consists of
all $k$-intervals with $k\le r/2$ in $[n]$ and all $(r/2+1)$-intervals
containing the element 1, while the vertex set of  $\Qant_{n,r+1}$ consists of
all $k$-intervals with $k\le r/2$ and all $(r/2+1)$-intervals containing $n$;
cf., e.g.~\cite[Prop.~16.3]{DKK3}. Here a $k$-\emph{interval} is meant to be
the union of $k$ intervals, but not $k'$ intervals with $k'<k$.

In particular, it follows that the set $\bfCsym_n$ of maximal symmetric chord
separated collections in $2^{[n]}$ with $n$ even forms a poset with unique
minimal and maximal elements, which are  formed by all intervals in $[n]$ and
all 2-intervals containing 1 (in the former case) and $n$ (in the latter case).


\section{A relation to higher Bruhat orders} \label{sec:Bruhat}

Manin and Schekhtman~\cite{MS} introduced higher Bruhat orders as a
generalization of the notion of a weak Bruhat order (being a partial order on
the symmetric group via inversions).  Recall some definitions from~\cite{MS}.
 \medskip

\noindent\textbf{Definitions.} Consider integers $n>d>0$. A set
$P\in\binom{[n]}{d+1}$ whose elements are ordered lexicographically is called a
\emph{packet} (of size $d+1$ in $[n]$). The \emph{family} $\Fscr(P)$ is meant
to be the ordered collection of all $d$-element subsets of $P$ where these
subsets follow in the lexicographic order (i.e., if $P=(i_1<\cdots<i_{d+1})$,
then $\Fscr(P)$ is $(P-i_{d+1},\, P-i_d,\ldots, P-i_1)$. A linear (total) order
$\rho$ on all $d$-element subsets of $[n]$ is called \emph{admissible} if its
restriction to each packet $P$ is either lexicographic or anti-lexicographic,
i.e., $\rho_{|_P}$ is either identical or reverse to the order on $\Fscr(P)$.
The set of admissible orders is denoted by $A(n,d)$. Two orders $\rho,\rho'\in
A(n,d)$ are called \emph{elementarily equivalent} if they differ by
interchanging two neighbors and these neighbors are not contained in the same
packet. The quotient of $A(n,d)$ by the equivalence relation is denoted by
$B(n,d)$, and the natural projection of $A(n,d)$ to $B(n,d)$ by $\pi$. The set
$\Inver(\rho)$ of \emph{inversions} of $\rho\in A(n,d)$ consists of the packets
$P\in\binom{[n]}{d+1}$ for which $\Fscr(P)$ has the reverse order in $\rho$.
For $r=\pi(\rho)\in B(n,d)$, we write $\Inver(r)$ for $\Inver(\rho)$ (this is
correct since $\Inver(\rho)=\Inver(\rho')$ when $\pi(\rho)=\pi(\rho')$). This
provides a partial order $\prec$ on $A(n,d)$ or on $B(n,d)$ (the latter is
called the \emph{Bruhat order} for $(n,d)$); here for $r,r'\in B(n,d)$, we
write $r\prec r'$ if there is a sequence $r=r_1,r_2,\ldots, r_N=r'$ such that
for each $i=2,\ldots,N$, ~$\Inver(r_{i-1})\subset \Inver(r_i)$ and
$|\Inver(r_i)-\Inver(r_{i-1})|=1$.
  \medskip

(In particular, $A(n,1)$ is the set of linear orders $\sigma$ on $[n]$,
$\Inver(\sigma)$ is the set of pairs $i<j$ with $\sigma(i)>\sigma(j)$
(inversions), and $(A(n,1),\prec)$ turns into the weak Bruhat order on the
symmetric group $\frakS_n$.)

Kapranov and Voevodsky~\cite{KV} and Ziegler~\cite{zieg} gave a nice geometric
interpretation of higher Bruhat orders. More precisely, the following is true
(using terminology and notation from Sect.~\SEC{rsep-cubil}).
  \begin{numitem1} \label{eq:Bruhat-cub}
There is a bijection $\tau$ of $B(n,d)$ to the set $\bfQ_{n,d}$ of cubillages
on $Z(n,d)$ or, equivalently, to the set of (abstract) membranes $M$ in
$Z(n,d+1)$. Here $r\in B(n,d)$ is mapped to $M=\tau(r)$ if $\Inver(r)$ is equal
to $\Inver(M)$ (the set of types of cubes in $Q(M)$ (i.e., lying before $M$)
for any cubillage on $Z(n,d+1)$ containing $M$). Under this correspondence,
$r\prec r'$ holds if and only if the membranes $\tau(r)$ and $\tau(r')$ belong
to the same cubillage $Q$ on $Z(n,d+1)$ and the former is obtained from the
latter by a series of lowering flips using cubes of $Q$.
  \end{numitem1}

Now we are going to define a sort of symmetric Bruhat orders as follows.
\medskip

\noindent\textbf{Definitions.} For a packet $P=(p_1\prec\cdots\prec p_{d+1})$
in $[n]$, define its symmetric packet to be $P^\circ=(p_{d+1}^\circ\prec \cdots
\prec p_1^\circ)$. Accordingly, the family $\Fscr(P^\circ)$ is regarded as
symmetric to $\Fscr(P)$. We say that a linear order $\rho$ on $\binom{[n]}{d}$
is \emph{s-admissible} if for each packet $P\in\binom{[n]}{d+1}$, the
restrictions of $\rho$ to (the $d$-element subsets of) $P$ and $P^\circ$ are
either both lexicographic or both anti-lexicographic. The set of s-admissible
orders is denoted by $A^s(n,d)$. Orders $\rho,\rho'\in A^s(n,d)$ are called
\emph{elementarily equivalent} if they differ by interchanging two neighbors
not contained in the same packet and, simultaneously, by interchanging its
symmetric neighbors. The quotient of $A^s(n,d)$ by the equivalence relation is
denoted by $B^s(n,d)$, and the natural projection of $A^s(n,d)$ to $B^s(n,d)$
by $\pi^s$. This gives a symmetric set $\Inver(\rho)=\Inver(r)$ of (pairs of)
inversions, where $\rho\in A^s(n,d)$ and $r=\pi^s(\rho)$. The partial order
$\prec^s$ on $A^s(n,d)$ or on $B^s(n,d)$ is defined accordingly. So for
$r,r'\in B^s(n,d)$, we write $r\prec r'$ if there is a sequence
$r=r_1,\ldots,r_N=r'$ such that for each $i$, $\Inver(r_i)$ is obtained from
$\Inver(r_{i-1})$ by adding either one self-symmetric inversion or a pair of
symmetric ones. We call $(A^s(n,d),\prec^s)$ (or $(B^s(n,d),\prec^s$) the
\emph{symmetric Bruhat order} for $(n,d)$.
  \medskip

(Note that, acting in a somewhat similar fashion, we could attempt to formally
define a sort of \emph{skew-symmetric} Bruhat orders. Here one should exclude
from consideration the self-symmetric packets and think of \emph{admissible}
orders as those linear orders on $\binom{[n]}{d}$ whose restriction to each
packet $P$ is lexicographic if and only if the restriction to $P^\circ$ is
anti-lexicographic. But this stuff is beyond our paper.)
 \medskip

When $n$ is odd (even), $B^s(n,d)$ may be interpreted as the Bruhat order of
\emph{type B} (resp. \emph{type C}) for $(n,d)$. For an extensive discussion on
Bruhat orders of type B and C and their implementations, see~\cite{DKK6}.

Our constructions and results in Sect.~\SEC{symm-r-even} lead to the following

 \begin{theorem} \label{tm:BruhatC}
When $n$ is even and $d$ is odd, there is a bijection $\sigma$ between
$B^s(n,d)$ (of type C) and the set $\bfQsym_{n,d}$ of symmetric cubillages on
$Z(n,d)$, or, equivalently, the set of (abstract) symmetric membranes $M$ in
$Z(n,d+1)$. Here $r\in B^s(n,d)$ corresponds to $M=\sigma(r)$ if
$\Inver(r)=\Inver(M)$. Under this correspondence, $r\prec r'$ holds if and only
if $\sigma(r)$ and $\sigma(r')$ belong to the same symmetric cubillage on
$Z(n,d+1)$ and the former membrane is obtained from the latter by a series of
symmetric (double or central) lowering flips.
 \end{theorem}

  \Xcomment{
Note also that if we deal with the class of maximal $\circ$-symmetric
s-collections on $2^{[n]}$ and, accordingly, with the set
$\circ\mbox{-}\bfQsym_{n,2}$ of $\circ$-symmetric tilings (see Remark~3 in
Sect.~\SEC{scoll}), then the corresponding flip graph determines the Bruhat
order $B^s(n,2)$, as is explained in~\cite{DKK6}.
  }


\section{Interrelations between symmetric cubillages and membranes}
\label{sec:gen-theorems}

We know that each zonotope $Z(n,d)$ contains a cubillage (in particular, the
standard and anti-standard ones), and that any cubillage in $Z(n,d)$ is lifted
as a membrane in some cubillage on $Z(n,d+1)$ (by~\refeq{membr-in-zon}). This
section gives symmetric versions of those properties when the number $n$ of
colors is even (regarding the $\ast$-symmetry).

 \begin{theorem} \label{tm:cub-membr}
For $n$ even, let $Q$ be a symmetric cubillage on $Z(n,d)$. Then $Q$ contains a
symmetric membrane.
   \end{theorem}
   \begin{theorem} \label{tm:membr-cub}
For $n$ even, let $Q$ be a symmetric cubillage on $Z(n,d)$. Then there exists a
symmetric cubillage $Q'$ on $Z(n,d+1)$ and a membrane $M$ of $Q'$ isomorphic to
$Q$, i.e., $M=M_Q$.
  \end{theorem}

\noindent\textbf{Proof of Theorem~\ref{tm:cub-membr}.}
When $d$ is even, the
theorem follows from assertion~(i) in Lemma~\ref{lm:fr-rear}. Now assume that
$d$ is odd, and let $Q$ be a symmetric cubillage on $Z=Z(n,d)$. To obtain the
desired symmetric membrane in $Q$, we construct, step by step, a sequence of
pairs of symmetric membranes $M,M^\ast$ starting with the pair $\Zfr,\Zrear$
(which are symmetric to each other by Lemma~\ref{lm:fr-rear}(ii)).

For a membrane $M$, consider the set $Q(M)$ of cubes of $Q$  between $\Zfr$ and
$M$ (see Sect.~\SSEC{invers}).  We assume that the current membranes $M$ and
$M^\ast$ satisfy $Q(M)\subseteq Q(M^\ast)$.

If $Q(M)=Q(M^\ast)$, then $M=M^\ast$, and we are done. Now assume that
$Q(M^\ast)$ strictly includes $Q(M)$. Among the cubes in $Q(M^\ast)-Q(M)$,
choose a minimal cube $C$ (w.r.t. the natural order $\prec_Q$). Then $\Cfr$ is
entirely contained in $M$. Replacing in $M$ the disc $\Cfr$ by $\Crear$, we
obtain membrane $M'$ of $Q$ with $Q(M')=Q(M)\cup\{C\}$.

Since $Q$ is symmetric, it has the cube $C^\ast$ symmetric to $C$. By the
symmetry of $M$ and $M^\ast$, the side $C^{\ast\rm{rear}}$ (which is symmetric
to $\Cfr$) is entirely contained in $M^\ast$, and $C^\ast$ belongs to
$Q(M^\ast)-Q(M)$. The oddness of $d$ implies that the cubes $C$ and $C^\ast$
are different (for if $C=(X\,|\, T)$, then $C^\ast=((XT)^\ast\,|\, T^\circ)$,
and $T=T^\circ$ is impossible since $d$ is odd and $n$ is even).

Thus, replacing in $M^\ast$ the side $C^{\ast\rm{rear}}$ by  $C^{\ast\rm{fr}}$,
we obtain the membrane ${M'}^\ast$ symmetric to $M'$, and moreover,
$Q(M')\subseteq Q({M'}^\ast)$ is valid. Also the gap between $M'$ and
${M'}^\ast$ becomes smaller. Repeating the procedure, we eventually obtain a
membrane coinciding with its symmetric one, as required. \hfill\qed

 \begin{corollary} \label{cor:zon-sym-cub}
For any $d\le n$ with $n$ even, the zonotope $Z(n,d)$ has a symmetric
cubillage. Therefore, the set $\bfSsym_{n,d-1}$ of symmetric $(d-1)$-separated
collections of size $s_{n,d-1}$ in $2^{[n]}$ is nonempty.
  \end{corollary}

This is shown by induction on $d$ (by decreasing $d$). The zonotope $Z(n,n)$
(which is a single cube) is symmetric. Suppose that for $d<n$, the zonotope
$Z=Z(n,d+1)$ has a symmetric cubillage $Q$. By Theorem~\ref{tm:cub-membr}, $Q$
contains a symmetric membrane $M$. Then the projection $\pi(M)$ is a symmetric
cubillage on $Z(n,d)$.
 \medskip

\noindent\textbf{Proof of Theorem~\ref{tm:membr-cub}.}
Take the (abstract)
membrane $M=M_Q$ in $Z'=Z(n,d+1)$. This $M$ is self-symmetric.

If $d+1$ is odd, then in order to construct the desired symmetric cubillage on
$Z'$ containing $M$, we first choose an arbitrary cubillage $R$ on $Z'$
containing $M$. Take the set $R(M)$ of cubes of $R$ between $\Zpfr$ and $M$.
Since $M$ is self-symmetric and $\Zprear$ is symmetric to $\Zpfr$, by
Lemma~\ref{lm:fr-rear}(ii) (with $d+1$ instead of $d$), the cubes symmetric to
those in $R(M)$ should be disposed in the region between $M$ and $\Zprear$, and
moreover, they give a proper subdivision of this region. Hence
$R(M)\cup\{C^\ast\colon C\in R(M)\}$ is a symmetric cubillage containing $M$,
as required.

Now let $d+1$ be even. Then $d$ is odd, and using
Corollary~\ref{cor:flip-rsepar}, we can construct a sequence of central or
double lowering flips in $Q$ so as to reach the standard cubillage $\Qst_{n,d}$
on $Z=Z(n,d)$. Each central flip uses a self-symmetric capsid $\frakD$ with the
anti-standard filling $\frakDst$; this corresponds to a self-symmetric cube $C$
in $Z'$ and the flip $\frakDant\leadsto \frakDst$ determines the flip
$\Crear\leadsto \Cfr$ in $Z'$. In its turn, each double flip uses a pair of
innerly disjoint symmetric capsids $\frakD,\frakD^\ast$, both with the
anti-standard fillings. They correspond to (different) symmetric cubes
$C,C^\ast$, and the flip $(\frakDant,\frakD^{\ast\rm{ant}})\leadsto
(\frakDst,\frakD^{\ast\rm{st}})$ determines the double lowering flip
$(\Crear,C^{\ast\rm{rear}})\leadsto(\Cfr,C^{\ast\rm{fr}})$ in $Z'$.

Thus, the above sequence of flips in $Z$ starting with $Q$ and terminating with
$\Qst_{n,d}$ determines a symmetric set of cubes, denoted as $Q^-$, which
subdivide the region between $\Zpfr$ and $M$. Acting similarly with symmetric
raising (central or double) flips in $Z$ starting with $Q$ and ending with
$\Qant_{n,d}$, we can construct a symmetric set $Q^+$ of cubes filling the
region between $M$ and $\Zprear$. Now $Q^-\cup Q^+$ is the desired symmetric
cubillage in $Z'$ containing $M$, as required. \hfill \qed


\section{Concluding Remarks} \label{sec:concl}

Besides the $\ast$-symmetry (given in~\refeq{Xast}), one can deal with the
$\circ$-\emph{symmetry} on $2^{[n]}$, also called the \emph{color exchanging
symmetry}, which sends $X\subseteq [n]$ to $X^\circ:=\{i\in[n]\colon i^\circ\in
X\}$. There are interesting interrelations between both sorts of symmetry. One
of them is that for a $\ast$-symmetric cubillage $Q$ on $Z(n,d)$, each cube
$C=(X\,|\,T)$ is symmetric to the cube $C^\ast$ viewed as
$((XT)^\ast\,|\,T^\circ)$ (see Sect.~\SSEC{sym_cub}); so under the
$\ast$-symmetry on $Q$, the types of cubes obey the $\circ$-symmetry relation.

When $n$ is even, $n=2m$, there is a bijection $\omega$ on $2^{[n]}$ that turns
$\ast$-symmetric sets into $\circ$-symmetric ones. More precisely, for
$X\subseteq [n]$, let
$$
  X_-:=X\cap[m] \qquad\mbox{and}\qquad X_+:=X\cap[m+1..2m].
  $$

For $i\in[2m]$, define
     \begin{equation} \label{eq:ilozenge}
 i^\dsymm:= \left\{
  \begin{array}{rl}
 m-i+1 & \quad \mbox{if $i\le m$,} \\
 2m-i+1 & \quad \mbox{if $m<i\le 2m$,}
   \end{array}
  \right.
   \end{equation}
and accordingly extend this to subsets $X\subseteq[2m]$ by setting
$X^\dsymm:=\{i^\dsymm \colon i\in X\}$. Note that $\circ$ and $\dsymm$ commute:
$(i^\circ)^\dsymm=(i^\dsymm)^\circ$.

Now $X\mapsto Y=\omega(X)$ is defined by
  \begin{equation} \label{eq:X-to-Y}
  Y_-=[m]-(X_-)^\dsymm=([m]-X_-)^\dsymm\qquad \mbox{and} \qquad Y_+=(X_+)^\dsymm.
  \end{equation}

  \begin{lemma} \label{lm:sym_to_sym}
Let $X\subseteq [n]$ and $Y=\omega(X)$. Then $\omega(X^\ast)=Y^\circ$.
  \end{lemma}
  \begin{proof}
Let $i\in[m]$. We examine four cases of $\{i,i^\circ\}$ relative to $X$.
 \smallskip

\noindent\emph{Case 1}: $i\in X\not\ni i^\circ$. Then $i\in X^\ast\not\ni
i^\circ$. Applying~\refeq{X-to-Y} to $X$ and to $X^\ast$, we have
$i^\dsymm\notin \omega(X)$ and $(i^\circ)^\dsymm\notin \omega(X)$, and
similarly $i^\dsymm,(i^\circ)^\dsymm\notin \omega(X^\ast)$.
 \smallskip

\noindent\emph{Case 2}: $i\notin X\ni i^\circ$. Then $i\notin X^\ast\ni
i^\circ$. It follows that $i^\dsymm\in \omega(X)$ and $(i^\circ)^\dsymm\in
\omega(X)$, and similarly $i^\dsymm,(i^\circ)^\dsymm\in \omega(X^\ast)$.
 \smallskip

\noindent\emph{Case 3}: $i,i^\circ\in X$. Then $i,i^\circ\notin X^\ast$. It
follows that $i^\dsymm\notin \omega(X)\ni (i^\circ)^\dsymm$ and that
$i^\dsymm\in \omega(X^\ast)\not\ni (i^\circ)^\dsymm$.
 \smallskip

\noindent\emph{Case 4}: $i,i^\circ\notin X$. Then $i,i^\circ\in X^\ast$,
yielding $i^\dsymm\in \omega(X)\not\ni (i^\circ)^\dsymm$ and $i^\dsymm\notin
\omega(X^\ast)\ni (i^\circ)^\dsymm$.
  \smallskip

In all cases, $\omega(X)$ and $\omega(X^\ast)$ are $\circ$-symmetric within the
pair $\{i^\dsymm,(i^\dsymm)^\circ)\}$, whence the result follows.
 \end{proof}

Thus, $\omega$ maps $\ast$-symmetric collections into $\circ$-symmetric ones.
One can see that the converse holds as well. The next lemma involves separation
relations.

 \begin{lemma} \label{lm:r-separ_preserv}
Let $n=2m$ be even (as before) and $r$ odd. Let $A,B\subseteq[n]$ be
$r$-separated. Then $C:=\omega(A)$ and $D:=\omega(B)$ be $r$-separated as well.
Conversely, if $C,D$ are $r$-separated, then so are $A,B$.
   \end{lemma}
  \begin{proof}
Sets $X,Y\subseteq[n]$ are said to be $k$-\emph{intertwined} if $k$ is the
minimal number such that there are elements $i_1<\cdots<i_k$ of $[n]$ that
alternate in $X-Y$ and $Y-X$; we denote this $k$ as $\iota(X,Y)$.

One can see that $\iota(A_-,B_-)+\iota(A_+,B_+)$ is equal to either
$\iota(A,B)$ or $\iota(A,B)+1$. Also one can see that
  \begin{gather*}
\iota(C_-,D_-)=\iota([m]-A_-^\dsymm,\,[m]-B_-^\dsymm)
      =\iota(A_-^\dsymm,B_-^\dsymm)=\iota(A_-,B_-); \\
  \iota(C_+,D_+)=\iota(A_+^\dsymm,B_+^\dsymm)=\iota(A_+,B_+); \quad\mbox{and} \quad
  \iota(C,D)\le \iota(C_-,D_-)+\iota(C_+,D_+).
  \end{gather*}

This implies that if $\iota(A,B)\le r$ or if
$\iota(A,B)=r+1=\iota(A_-,B_-)+\iota(A_+,B_+)$, then $\iota(C,D)\le r+1$, and
therefore $C,D$ are $r$-separated.

So assume that $i(A,B)=r+1$ and $\iota(A_-,B_-)+\iota(A_+,B_+)=\iota(A,B)+1$.
The latter equality is possible only if both $p:=\max(A_-\triangle B_-)$ and
$q:=\min(A_+\triangle B_+)$ belong to the same set among $A-B$ and $B-A$; let
for definiteness $p,q\in A-B$.

The transformations $A_-\mapsto[m]-A_-$ and $B_-\mapsto [m]-B_-$ swaps the
alternating pieces of $A_--B_-$ and $B_--A_-$. This implies that
$\max(([m]-A_-)\triangle([m]-B_-))$ is equal to $p$ and that $p\in [m]-B_-$.
Then
  $$
  p^\dsymm=\min(([m]-A_-)^\dsymm\triangle([m]-B_-)^\dsymm)=\min(C_-\triangle
  D_-)\in D_--C_-
  $$
(since $p\in[m]-B_-$ implies $p^\dsymm\in([m]-B_-)^\dsymm$).

At the same time, under the transformations $A_+\mapsto A_+^\dsymm$ and
$B_+\mapsto B_+^\dsymm$, the minimal element $q$ of $A_+\triangle B_+$ maps to
the maximal element $q^\dsymm$ of $A_+^\dsymm\triangle B_+^\dsymm=C_+\triangle
D_+$, and this $q^\dsymm$ belongs to $A_+^\dsymm-B_+^\dsymm=C_+-D_+$ (since
$q\in A_+-B_+$).

Thus, $\min(C\triangle D)=p^\dsymm\in D-C$ while $\max(C\triangle
D)=q^\dsymm\in C-D$, implying that $\iota(C,D)$ is even. Since $r$ is odd and
$\iota(C,D)\le i(A,B)+1=r+2$, we can conclude that $\iota(C,D)=r+1$, and
therefore $C,D$ are $r$-separated.

The converse assertion is shown by reversing the above reasonings.
  \end{proof}

The above lemmas imply the following
  \begin{corollary} \label{cor:separ-relat}
For $n$ even and $r$ odd, if $\Sscr$ is a $\ast$-symmetric $r$-separated
collection in $2^{[n]}$, then $\omega(\Sscr):=\{\omega(X)\colon X\in\Sscr\}$ is
a $\circ$-symmetric $r$-separated collection, and vice versa.
  \end{corollary}

This gives a bijection between the max-size $\ast$-symmetric and
$\circ$-symmetric collections in $2^{[n]}$, leading to a one-to-one
correspondence between $\ast$-symmetric and $\circ$-symmetric cubillages on
$Z=Z(n,d=r+1)$ (when both $n,d$ are even).

More precisely, the correspondence $Q\mapsto Q'=:\omega(Q)$ of such cubillages
(where $Q$ is $\ast$-symmetric and $Q'$ is $\circ$-symmetric) is given via the
relation $V_{Q'}=\omega(V_Q)$ on their spectra, the sets of vertices regarded
as collections in $2^{[n]}$. (Here we use the facts that $\omega(V_Q)$ is a
max-size $(d-1)$-separated collection and that each max-size $(d-1)$-separated
collection in $2^{[n]}$ forms a spectrum of a cubillage on $Z$.)

One can check that for vertices $A,B$ of $Q$, $|A\triangle B|=1$ implies
$|\omega(A)\triangle \omega(B)|=1$, and vice versa. Equivalently, vertices
$A,B$ of $Q$ are connected by edge if and only if so are the vertices
$\omega(A)$ and $\omega(B)$ of $Q'$. This is extended to the faces of $Q$ and
$Q'$, namely, if $F=(X\,|\,T)$ is a face of $Q$, then $(\omega(X)\,|\,
T^\dsymm)$ is a face of $Q'$, denoted as $\omega(F)$, and conversely, for a
face $(X'\,|\,T')$ of $Q'$, $(\omega^{-1}(X')\,|\,T'^{\,\dsymm})$ is a face of
$Q$.

Therefore, the $\ast$- and $\circ$-symmetric cubillages are, in fact,
represented by the same complex (by ignoring the directions of edges). The
correspondence $Q\mapsto Q'$ has a nice visualization when $d=2$; namely, $Q'$
is mirror-reflected to $Q$ w.r.t. the SW-to-NE line (at angle of $45^\circ$)
through the center of $Z(n,2)$.

In general, the bijection on the cubes is extended, in a natural way, to the
capsids of $Q$ and $Q'$. However, the fillings of corresponding capsids may be
different: if a capsid $\frakD$ of $Q$ has the standard filling, say, then the
capsid $\omega(\frakD)$ of $Q'$ may have any of the two possible fillings.

One can see that for a cube $C=(X\,|\,T)$ in $Z$, its $\circ$-symmetric cube
$C^\circ$ is viewed as $(X^\circ\,|\, T^\circ)$. This easily implies that
  \begin{numitem1} \label{eq:capsids_relat}
for a capsid $\frakD$ in $Z$, both $\frakD$ and its $\circ$-symmetric capsid
$\frakD^\circ$ have fillings of the same type: both are either standard or
anti-standard.
  \end{numitem1}

When all capsids of $Q'$ have standard (resp. anti-standard) fillings, we just
deal with the standard (resp. anti-standard) cubillages on $Z$, which give
important special cases of $\circ$-symmetric cubillages.

From a viewpoint of capsids, the difference between the $\ast$- and
$\circ$-symmetric settings is impressive for $d=2$ where, in a $\ast$-symmetric
cubillage, any pair of symmetric capsids $\frakD$ and $\frakD^\ast$ have
different fillings (whereas the fillings of their symmetric counterparts
$\omega(\frakD)$ and $\omega(\frakD^\ast)$ are similar).


\appendix


\section{Appendix 1: Symmetric $r$-separated collections and
$(r+1)$-dimensional cubillages when $r$ is odd.}  \label{sec:symm-r-odd}

In this additional section we assume, as before, that the number $n$ of colors
is even, and are going to explore the flip structure in the set $\bfSsym_{n,r}$
of size-maximal (viz. representable) symmetric $r$-separated collections in
$2^{[n]}$ when $r$ is \emph{odd}. An attempt to show the connectedness of this
structure looks more intricate than in the case of $r$ even considered in
Sect.~\SEC{symm-r-even}. Again, we essentially attract a machinery of
cubillages.

Consider a collection $\Sscr\in\bfSsym_{n,r}$ and the corresponding symmetric
cubillage $Q$ with $V_Q=\Sscr$ on $Z=Z(n,d)$, where $d:=r+1$ is even. We may
assume that $d\ne n$ (since $\bfSsym_{n,n-1}=\{2^{[n]}\}$). Usual flips in
cubillages are performed by using capsids (see Sect.~\SSEC{memb-capsid}), and
we are going to apply them to construct symmetric transformations.

So let $\frakD=(X\,|\, T)$ be a capsid with $T=(p_1<\cdots<p_{d+1})$ in $Q$ and
assume for definiteness that it has the standard filling $\frakDst$. This
filling is formed by $\lceil d/2\rceil$ lower cubes $F_i$ (with $d-i$ odd) and
$\lfloor d/2\rfloor$ upper cubes $G_j$ (with $d-j$ even), namely:
  $$
  F_i=(X\,|\, T-p_i) \;\; \mbox{for $i=1,3,\ldots, d+1$},\;\; \mbox{and}\;\;
    G_j=(Xp_j\,|\, T-p_j)\;\;\mbox{for $j=2,4,\ldots,d$}
    $$
(cf.~\refeq{capsid}(i)). Then the capsid $\frakD^\ast$ symmetric to $\Dscr$
(existing since $Q$ is symmetric) has the anti-standard filling formed by the
$\lceil d/2\rceil$ upper cubes
 $$
F_i^\ast=((XT)^\ast p_i^\circ\,|\, T^\circ-p_i^\circ)=:
        G'_{d+2-i}=(X'p_i^\circ\,|\, T^\circ-p_i^\circ), \;\;\;   i=1,3,\ldots,d+1,
   $$
(relabeling $(XT)^\ast$ as $X'$) and $\lfloor d/2\rfloor$ lower cubes
 $$
G_j^\ast=((XT)^\ast\,|\, T^\circ-p_j^\circ))=: F'_{d+2-j}=(X'\,|\, T^\circ-p_j^\circ), \;\;\;
         j=2,4,\ldots,d.
   $$

It follows that $\frakD\ne\frakD^\ast$. Moreover, since  the Claim in the proof
of Lemma~\ref{lm:C-Cast} is valid for any $n$ and $d<n$, only two cases are
possible:
  \begin{numitem1} \label{eq:a-b}
$\frakD$ and $\frakD^\ast$ share either (a) no cube, or (b) exactly one cube.
  \end{numitem1}

In case~(a), we can make the \emph{double (symmetric) flip} in $Q$ using
$\frakD,\frakD^\ast$, by replacing the filling $\frakDst$ by $\frakDant$, and
$\frakD^{\ast \rm{ant}}$ by $\frakD^{\ast \rm{st}}$. This results in another
symmetric cubillage on $Z$. (Note that, in contrast to flips in
Sect.~\SEC{symm-r-even}, the double flip is now viewed as ``undirected'' since
it is raising in $\frakD$ but lowering in $\frakD^\ast$, as it was demonstrated
for the special case $d=r+1=2$ described in Sect.~\SEC{scoll}.)

Now suppose that we are in case~(b) of~\refeq{a-b}. Then $\frakD\cap
\frakD^\ast$ is a self-symmetric cube $C=(\hat X\,|\,\hat T)$ whose type (color
set) $\hat T$ is symmetric and occurs in both $T$ and $T^\circ$. Then $\hat
T=T-p_k=T^\circ-p_k^\circ$ for some $k\in[d+1]$, and either
  \medskip

(i) $d-k$ is odd, $C=F_k=G'_{d+2-k}$ and $X=X'p_k^\circ$, or
  \smallskip

(ii) $d-k$ is even, $C=G_k=F'_{d+2-k}$ and $X'=Xp_k$.
 \smallskip

We call $X'$ in case~(i) and~$X$ in case~(ii) the \emph{bottom} of $\frakD\cup
\frakD^\ast$, denoted by $\tilde X$, and denote $T\cup T^\circ$ by $\tilde T$.
Then $\tilde T$ is symmetric and $|\tilde T|=d+2$ (which is even).

The important special case arises when the union of $\frakD$, $\frakD^\ast$ and
some extra cubes of $Q$ forms a subzonogon $\frakB$  with the bottom $\tilde X$
and the color set $\tilde T$. In other words, $\frakB$ is isomorphic to
$Z(\{\xi_\alpha\colon \alpha \in\tilde T\})\simeq Z(d+2,d)$. We call $\frakB$ a
\emph{barrel} of $Q$ (or a barrel in $Z$ compatible with $Q$). Also we refer to
the set of cubes of $Q$ occurring in $\frakB$ as the \emph{filling} of
$\frakB$, and the cube $C$ as its \emph{central cube}.
\medskip

\noindent\textbf{Remark 3.} ~Since $d$ and $d+2$ are even, the filling of
$\frakB$ forms a symmetric subcubillage of $Q$. By Theorem~\ref{tm:membr-cub},
any symmetric cubillage on $Z(d+2,d)$ is the projection of some membrane in
some symmetric cubillage $Q'$ on $Z(d+2,d+1)$. The latter zonotope is, in fact,
a capsid; it admits only two cubillages: the standard and anti-standard ones,
say, $Q_1,Q_2$ (see Sect.~\SSEC{memb-capsid}). Both of them are self-symmetric
(as being the projections of the front and rear sides of the cube $Z(d+2,d+2)$;
cf. Lemma~\ref{lm:fr-rear}). By Theorem~\ref{tm:cub-membr}, each of $Q_1,Q_2$
has a symmetric membrane, and one can conclude from~\refeq{order_capsid} (with
$d+2$ instead of $d+1$) that each $Q_i$ has exactly one symmetric membrane,
namely, the one dividing the sequence in (i) or (ii) of~\refeq{order_capsid}
half-to-half. Hence the barrel $\frakB$ admits two symmetric fillings, one
coming from a membrane of $Q_1$, and another from $Q_2$.
  \medskip

\noindent\textbf{Example.} Let $d=2$ and $n=4$. Then $[n]=\{1,2,3,4\}$,
$1^\circ=4$ and $2^\circ=3$. The zonotope (zonogon) $Z(4,2)$ is the simplest
barrel $\frakB$. It has two symmetric fillings (tilings) $B_1$ and $B_2$. Here
$B_1$ is the projection of the (unique) symmetric membrane $M_1$ of the
standard filling $\frakDst$ of the capsid (zonotope) $\frakD=Z(4,3)$, while
$B_2$ is the projection of the symmetric membrane $M_2$ in  $\frakDant$. The
cubillage $\frakDst$ consists of the cubes $F_4\prec G_3\prec F_2\prec G_1$,
having types $123,\,124,\,134,\,234$, respectively, and $M_1$ divides these
cubes half-to-half. This gives four facets of $M_1$, namely: $H_1:=F_4\cap
G_1$, $H_2:=G_3\cap F_2$, $H_3:=F_4\cap F_2$, $H_4:=G_3\cap G_1$. These facets
have types 23,\,14,\,13,\,24, respectively, and their projections generate four
cubes (rhombuses) in $B_1$; the first two rhombuses have symmetric types 23 and
14 and are ordered as $\pi(H_1)\prec \pi(H_2)$ (by the natural partial order in
$B_1$). Also $M_1$ has two more facets, contained in the boundary of $\frakD$
and having types 12 and 34. Altogether, we obtain 4+2=6 rhombuses in $B_1$. On
the other hand, one can check that the tiling $B_2$ (generated by $M_2$ in
$\frakDant$ with the cubes $F_1\prec G_2\prec F_3\prec G_4$) is formed by
another six-tuple of rhombuses $\pi(H'_i)$, in which the rhombuses
$\pi(H'_1),\pi(H'_2)$ having symmetric types 23 and 14 (respectively) are
ordered as $\pi(H'_1)\succ \pi(H'_2)$. Such barrel fillings are illustrated in
Fig.~\ref{fig:n=4}.
  \medskip

\noindent\textbf{Definition.} We call the replacement of one symmetric filling
of a barrel $\frakB$ of $Q$ by the other one a \emph{big} (or \emph{barrel})
\emph{flip} in $Q$ using $\frakB$, borrowing terminology from
Sect.~\SEC{scoll}.
 \medskip

Let $\Gammasym_{n,r}$ denote the undirected graph whose vertex set is
$\bfSsym_{n,r}$ and whose edges are related to the pairs of  symmetric
cubillages on $Z(n,d=r+1)$ where one is obtained from the other by either a
double flip or a big flip.

One can see that $\Gammasym_{n,1}$ is nothing else than the graph $\Gammasym_n$
(up to discarding the orientation of edges) constructed in Sect.~\SEC{scoll}
(see Theorem~\ref{tm:all-s-flips}). The big flips in it are just those as
illustrated in Fig.~\ref{fig:n=4} (cf. Example above).

We study the connected components of $\Gammasym_{n,r}$, trying to show that the
entire graph is connected, as follows. Let $\bfQsym_{n',d'}$ denote the set of
symmetric cubillages, and $\bfMsym_{n',d'}$ the set of (abstract) symmetric
membranes in $Z(n',d')$.

We know that the cubillages $Q\in\bfQsym_{n,d}$  one-to-one correspond to the
membranes $M\in\bfMsym_{n,d+1}$ (this bijection is given by $M\mapsto Q=\pi(M)$
and $Q\mapsto M=M_Q$); so we may concentrate on handling such membranes and
related flips on them. For $M\in\bfMsym_{n,d+1}$, let $\Kscr(M)$ denote the set
of symmetric cubillages on $Z'=Z(n,d+1)$ that contain $M$, and for
$K\in\bfQsym_{n,d+1}$, let $\Mscr(K)$ denote the set of symmetric membranes
contained in $K$. We say that $K$ and $M$ are \emph{agreeable} if
$K\in\Kscr(M)$ (and $M\in\Mscr(K)$). We introduce the following binary relation
on $\bfMsym_{n,d+1}$.
  \medskip

\noindent\textbf{Definition.} Two membranes $M,M'\in\bfMsym_{n,d+1}$ are called
\emph{equivalent} if there is a sequence $M=M_0,M_1,\ldots, M_N=M'$ of
membranes in $\bfMsym_{n,d+1}$ and a sequence $K_1,\ldots,K_N$ of cubillages in
$\bfQsym_{n,d+1}$ such that for each $i=1,\ldots,N$, both ~$M_{i-1},M_i$ belong
to $\Mscr(K_i)$ (and $K_{j},K_{j+1}\in\Kscr(M_j)$, $1\le j\le N-1$).
  \medskip

The equivalence relation is transitive and a maximal set of equivalent
membranes is called an \emph{orbit}. So $\bfMsym_{n,d+1}$ is partitioned into a
number of orbits.

 \begin{lemma} \label{lm:orbit}
Let $M,M'$ belong to the same orbit. Then $\pi(M)$ and $\pi(M')$ are connected
by a series of symmetric double flips.
  \end{lemma}
  \begin{proof}
It suffices to assume that $M,M'$ belong to the same set $\Mscr(K)$,
$K\in\bfQsym_{n,d+1}$. As in Sect.~\SSEC{invers}, we associate with $M$ the set
$K(M)$ of cubes of $K$ lying between $\Zpfr$ and $M$, and similarly for $M'$.
From the symmetry of $M,M'$ it follows that if $C\in K(M')-K(M)$, then
$C^\ast\in K(M)-K(M')$, and if $\Cfr\subset M$, then $\Castrear\subset M$.
Then, using the natural order on the cubes of $K$, one can find a sequence of
$k:=|K(M')-K(M)|$ cubes $C_1,\ldots,C_k$ and a sequence of $k+1$ membranes
$M=M_0,M_1,\ldots,M_k=M'$ in $\Mscr(K)$ such that for each $i$,
  $$
  K(M_i)-K(M_{i-1})=\{C_i\}\quad\mbox{and} \quad K(M_{i-1})-K(M_i)=\{C^\ast_i\}.
  $$
Then $M_i$ is obtained from $M_{i-1}$ by a symmetric cubic flip using
$C_i,C^\ast_i$, namely, by replacing the front side (disc) $\Cfr_i$ by the rear
side $\Crear_i$, and simultaneously, by replacing $\Castrear_i$ by $\Castfr_i$.
Therefore, the cubillage $Q_i:=\pi(M_i)$ on $Z$ is obtained from
$Q_{i-1}:=\pi(M_{i-1})$ by a symmetric double flip using the capsids formed by
the projections by $\pi$ of the corresponding sides of $C_i$ and $C^\ast_i$.
\end{proof}

(In fact, the orbits of $\bfMsym_{n,d+1}$ are analogous to blocks in
Sect.~\SEC{scoll}, where a block is the set of rhombus tilings agreeable with a
fixed permutation on $[m]$.)

Thus, we obtain the connectedness within each orbit, and now we come to the
problem of connecting different orbits. We try to do this by attracting special
membranes and capsids and making ``barrel flips'' in the projection of such
membranes.

More precisely, suppose that some cubillage $K\in\bfQsym_{n,d+1}$ contains a
self-symmetric capsid $\frakD$, and let $M$ be a membrane in $\Mscr(K)$. By the
symmetry of both $M$ and $\frakD$, the membrane $M$ must split $\frakD$ into
two symmetric halves (each containing $(d+2)/2$ cubes). Let $Q=\pi(M)$; then
$Q$ contains $\frac14(d+2)^2$ cubes coming from the facets in the interior $I$
of $M\cap\frakD$ (since any two cubes in a filling of $\frakD$ share a facet).

Assume, in addition, that besides the facets lying in $I$, $M$ contains a set
$J$ of facets of the \emph{boundary} of $\frakD$ so that the following property
holds: $M$ goes through the whole rim $\frakDfr\cap\frakDrear$ of $\frakD$; we
call such an $M$ \emph{perfect} w.r.t. $\frakD$. Then $I\cup J$ must have at
least $\binom{d+2}{d}$ facets, and $\pi(I\cup J)$ is nothing else than a barrel
$\frakB$ in $Q$. So we can make the big flip in $Q$ using $\frakB$ (as
described above). The resulting symmetric cubillage $Q'$ coincides with $Q$
outside $\frakB$, and therefore, the updated membrane $M'=M_Q$ coincides with
$M$ outside $\frakD$. Moreover, $M'$ must be agreeable with the cubillage $K'$
obtained from $K$ by the flip using $\frakD$.

Based on the above observations, we can argue as follows. Define $\bfDsym$ to
be the set of all (abstract) self-symmetric capsids $\frakD=(X\,|\,T)$ in
$Z'=Z(n,d+1)$ (i.e., $X,T\subset[n]$, $X\cap T=\emptyset$, $|T|=d+2$,
$X=(XT)^\ast$ and $T=T^\circ$). Also we associate with each orbit $\Oscr$ in
$\bfMsym_{n,d+1}$ the set $\Kscr(\Oscr):=\cup(\Kscr(M)\colon M\in \Oscr)$,
called the \emph{train} of $\Oscr$, and finish with the following conjecture.

  \begin{itemize}
\item[(C1):] Let $n,d$ be even. Then:
 \begin{itemize}
 \item[(i)] for each ``central'' capsid $\frakD\in\bfDsym$, there exists a cubillage
$K\in\bfQsym_{n,d+1}$ containing $\frakDant$ and a membrane $M\in \Mscr(K)$
that is perfect w.r.t. $\frakD$; and
 \item[(ii)] for each orbit $\Oscr$ in $\bfMsym_{n,d+1}$, the train $\Kscr(\Oscr)$
contains  either the standard cubillage on $Z'$, or a cubillage $K$ for which
there are $\frakD\in\bfDsym$ and $M\in\Mscr(K)$ such that $K$ contains
$\frakDant$ and $M$ is perfect w.r.t. $\frakD$.
  \end{itemize}
 \end{itemize}

One can realize that (C1) gives rise to a method that, starting with an
arbitrary membrane in $\bfMsym_{n,d+1}$ along with a cubillage agreeable with
$M$, updates, step by step, current membranes and cubillages (by making double
(cubic) flips on membranes preserving current cubillages or lowering central
(capsid) flips properly updating both the current cubillage and membrane) so as
to eventually reach the orbit whose train contains the standard cubillage on
$Z'$. As a consequence, the validity of~(C1) would provide the desired
property: for $n,d$ even, any two symmetric cubillages on $Z(n,d)$ could be
connected by a series of symmetric double or barrel flips, yielding the
connectedness of $\bfSsym_{n,d-1}$ by symmetric flips.


\section{Appendix 2: Symmetric $r$-separation in $2^{[n]}$ when $n$ is odd.}
\label{sec:n-odd}

Earlier we have described flip structures on $\bfSsym_{n,r}$ in $2^{[n]}$ when
$n$ is even. In this additional section we consider $\bfSsym_{n,r}$ when $n$ is
odd. Note that if, in addition, $r$ is odd, then this class may be empty; we
have seen this in Sect~\SSEC{sw-n-odd} for $r=1$, and suspect that a similar
behavior takes place for any odd $r$. (Recall that $\bfSsym_{n,r}$ consists of
those size-maximal $r$-separated collections $\Sscr$ in $2^{[n]}$ (i.e.,
satisfying $|\Sscr|=s_{n,r}$; see Sect.~\SEC{intr}) that are symmetric.)

In what follows we assume that $r$ is even (while $n$ is odd). First of all we
have to explain that in this case the set $\bfSsym_{n,r}$ is nonempty.
Equivalently, for $d:=r+1$, the set $\bfQsym_{n, d}$ of symmetric cubillages on
$Z(n,d)$ is nonempty (since $Q\mapsto V_Q$ gives a bijection between $\bfQ_{n,
d}$ and  $\bfS_{n,r}$). This is stated in Corollary~\ref{cor:sym-exist} below
and can be shown by using certain operations for cubillages on $Z(n,d)$ and
$Z(n-1,d)$, as follows.

It is convenient to assume that the set $\Xi$ of generating vectors of the
zonotope $Z=Z(\Xi)\simeq Z(n,d=r+1)$ is given in a symmetrized form. Let
$n=2m+1$. We relabel the colors in $[n]$ as $-m,\ldots,-1,0,1,\ldots, m$; then
for each $i$, the symmetric color $i^\circ$ is $-i$. When the generating
vectors are given as in~\refeq{symm_cyc_gen}, $\xi_0$ turns into the first unit
base vector $(1,0,\ldots,0)$.

Let $Z'\simeq Z(n-1,d)$ be the zonotope generated by $\Xi-\{\xi_0\}$. Consider
a symmetric cubillage $Q'$ on $Z'$ (existing since
$\bfSsym_{n-1,r}\ne\emptyset$, by Corollary~\ref{cor:zon-sym-cub}).
  \medskip

\noindent\textbf{Definitions.} Let $\pi^0$ denote the projection of $\Rset^d$
to $\Rset^{d-1}$ given by $x=(x(1),\ldots,x(d))\mapsto (x(2),\ldots,x(d))$
(where the coordinates of $\Rset^{d-1}$ are labeled $2,\ldots,d$). A
$(d-1)$-dimensional subcomplex $M$ of $Q'$ is called a \emph{0-membrane} if
$\pi^0$ homeomorphically maps $M$ (regarded as a subset of $\Rset^d$) onto
$Z^0:=\pi^0(Z')$. Such an $M$ subdivides $Z'$ into two closed regions
$\Zplow(M)$ and $\Zpup(M)$ formed by the points below and above $M$ (in the
direction of $\xi_0$), respectively; so $\Zplow(M)\cap\Zpup(M)=M$. Accordingly,
$\Qplow(M)$ and $\Qpup(M)$ denote the subcubillages of $Q'$ occurring in
$\Zplow(M)$ and $\Zpup(M)$, respectively.
  \medskip

There is a nice correspondence between cubillages and 0-membranes. It involves
two operations. The \emph{0-expansion operation} is applied to a pair
consisting of a cubillage on $Z'$ and a 0-membrane $M$ in $Q'$ and acts as
follows:
  \begin{itemize}
\item[(EXP):]
Move the set (subcubillage) $\Qpup(M)$ upward by $\xi_0$, keeping $\Qplow(M)$,
and fill the gap between $\Qplow(M)$ and $\Qpup(M)+\xi_0$ by cubes, each being
the Minkowsky sum of $F$ and $[0,\xi_0]$, where $F$ runs over the set of facets
in $M$.
  \end{itemize}

As a result, we obtain a cubillage on $Z=Z(n,d)$, called the \emph{0-expansion}
of $Q'$ using $M$ and denoted as $Q(Q',M)$.

Conversely, let $Q$ be a cubillage on $Z$ and let  $\Pi_0$ be the set of cubes
$C=(X\,|\,T)$ whose type $T$ contains color 0; this $\Pi_0$ is called the
\emph{0-pie} in $Q$ (adapting terminology in~\cite{DKK3}). The
\emph{0-contraction operation} applied to $Q$ acts as follows:
  \begin{itemize}
\item[(CON):]
Shrink each cube $C=(X\,|\,T)\in\Pi_0$ to its ``lower'' facet $F=(X\,|\,
T-\{0\})$, and for each cube $\tilde C=(\tilde X\,|\, \tilde T)$ of $Q$ whose
bottom $\tilde X$ contains color 0, move $\tilde C$ by $-\xi_0$, forming the
cube $(\tilde X-\{0\}\,|\, \tilde T)$ (preserving the remaining cubes of $Q$).
   \end{itemize}

One shows that the resulting set of cubes forms a cubillage $Q'$ on $Z'$, and
the set of facets obtained by shrinking the cubes of $\Pi_0$ forms a 0-membrane
$M$ in $Q'$; we call $(Q',M)$ the \emph{0-contraction} of $Q$. The 0-expansion
operation applied to $(Q',M)$ returns $Q$. This leads to the following relation
(cf.~\cite{DKK3}):
  \begin{numitem1} \label{eq:contr-expan}
the correspondence $(Q',M)\mapsto Q(Q',M)$ gives a bijection between the set of
pairs $(Q',M)$, where $Q'$ is a cubillage on $Z(n-1,d)$  and $M$ is a
0-membrane in $Q'$, and the set of cubillages on $Z(n,d)$.
  \end{numitem1}

Returning to symmetric settings as before, one can realize that (EXP) applied
to $Q'\in\bfQsym_{n-1,d}$ and a symmetric 0-membrane $M$ in $Q'$ produces a
symmetric cubillage on $Z(n,d)$, and conversely, (CON) applied to
$Q\in\bfQsym_{n,d}$ produces a symmetric cubillage on $Z(n-1,d)$ and a
symmetric 0-membrane in $Q'$. This yields a symmetric counterpart
of~\refeq{contr-expan}, namely:

  \begin{numitem1} \label{eq:sym-contr-expan}
the correspondence $(Q',M)\mapsto Q(Q',M)$ gives a bijection between the set of
pairs $(Q'\in \bfQsym_{n-1,d},M)$, where $M$ is a symmetric 0-membrane in $Q'$,
and the set $\bfQsym_{n,d}$.
  \end{numitem1}

By Corollary~\ref{cor:zon-sym-cub}, the set $\bfQsym_{n-1,d}$ is nonempty. A
similar fact for $(n,d)$ is provided by~\refeq{sym-contr-expan} and the
following

 \begin{lemma} \label{lm:sym-0-membr}
Let $n,d$ be odd and let $Q'$ be a symmetric cubillage on $Z'=Z(n-1,d)$. Then
$Q'$ contains a symmetric 0-membrane.
  \end{lemma}

  \begin{proof}
Such a membrane $M$ is constructed by a method similar to that in the proof of
Theorem~\ref{tm:cub-membr}.

More precisely, let $\Zpup$ ($\Zplow$) denote the upper (resp. lower) side of
the boundary of $Z'$ (which is formed by the points $x\in Z'$ such that there
is no $y\in Z'$ with $\pi^0(y)=\pi^0(x)$ and $y(1)>x(1)$ (resp. $y(1)<x(1)$)).
Both $\Zpup$ and $\Zplow$ are 0-membranes in $Q'$, and moreover, they are
symmetric to each other (to see the latter, one can use the fact that $\Zpfr$
and $\Zprear$ are symmetric to each other, in view of Lemma~\ref{lm:fr-rear}).

Starting with $(\Zpup, \Zplow)$, we construct, step by step, a sequence of
pairs of symmetric 0-membranes $(M,M^\ast)$ in $Q'$ such that $Q'(M)\supseteq
Q'(M^\ast)$, where $Q'(M')$ denotes that set of cubes of $Q'$ lying below a
0-membrane $M'$. When $Q'(M)= Q'(M^\ast)$, the current $M$ coincides with
$M^\ast$, and we are done.

So assume that $Q'(M)$ strictly includes $Q'(M^\ast)$. To construct the next
pair of 0-membranes, let us say that in $Q'$ a cube $C$ \emph{immediately
0-precedes} a cube $C'$ if $\Cup\cap \Cplow$ is a facet. Like the natural order
on the cubes of a cubillage (defined in Sect.~\SSEC{invers}), one shows that
the relation of immediately 0-preceding is free of directed cycles, and
therefore it determines a partial order on $Q'$, denoted as $\prec_0$. Note
that $C\prec_0 C'$ implies $C'^\ast\prec_0 C^\ast$. Take a maximal w.r.t.
$\prec_0$ cube $C$ in $Q'(M)-Q'(M^\ast)$. Then $\Cup$ is entirely contained in
$M$. By the symmetry, $(C^{\ast})^{\rm low}\subset M^\ast$. Moreover, the cubes
$C$ and $C^\ast$ are different (since $d$ is odd and the colors in $Q'$ are
partitioned into symmetric pairs). Now replacing $\Cup$ by $\Clow$ in $M$, and
$(C^\ast)^{\rm low}$ by $(C^\ast)^{\rm up}$ in $M^\ast$, we obtain a pair
$(M',M'^\ast)$ of symmetric 0-membranes for which $Q'(M')\supseteq Q'(M'^\ast)$
and the gap $Q'(M')-Q'(M'^\ast)$ becomes smaller. This yields the result.
  \end{proof}

 \begin{corollary} \label{cor:sym-exist} For $n,d$ odd, the set $\bfQsym_{n,d}$ is nonempty.
   \end{corollary}

Now we are going to devise  symmetric flips in $\bfQsym_{n,d}$ (and
$\bfSsym_{n,r}$). Consider a symmetric cubillage $Q$ on $Z=Z(n,d)$ and suppose
that it contains a capsid $\frakD=(X\,|\,T)$ having the anti-standard filling
$\frakDant$. Using the fact that $d$ is odd and arguing as in the proof of
Lemma~\ref{lm:stan-stan}, we observe that the symmetric capsid
$\frakD^\ast=((XT)^\ast\,|\,T^\circ)$ has the anti-standard filling as well.
Two cases are possible.
  \medskip

\noindent\emph{Case 1}: $\frakD$ and $\frakD^\ast$ have no cube in common. Then
we can apply to $Q$ the double lowering flip, by making the replacements
$\frakDant\leadsto \frakDst$ and $\frakDastant\leadsto \frakDastst$. This
results in another symmetric cubillage on $Z$.
  \medskip

\noindent\emph{Case 2}: $\frakD$ and $\frakD^\ast$ share a cube
$C=(X'\,|\,T')$. Note that $\frakD\ne\frakD^\ast$ (for otherwise $\frakD$ must
contain an edge of color 0 and the 0-contraction operation transforms $\frakD$
into a symmetric cube, which is impossible). Then $\frakD\cap\frakD^\ast=C$
(cf. the Claim in the proof of Lemma~\ref{lm:C-Cast}). Hence $C$ is
self-symmetric and $T'=T\cap T^\circ$. Since $|T'|=d$ is odd, $T'$ is of the
form $\{0\}\cup\{p_1,\ldots,p_t\}\cup \{-p_1,\ldots,-p_t\}$, where $t=(d-1)/2$
and $0<p_1<\cdots<p_t$. Also $T=T'\cup\{j\}$ and $T^\circ=T'\cup\{j^\circ=-j\}$
for some $j\ne 0$. One may assume that $X$ is the lowest vertex of
$\frakD\cup\frakD^\ast$ (then $|X|+1=|X'|=|(XT)^\ast|-1$).

An important special case arises when the union of $\frakD$, $\frakD^\ast$ and
some extra cubes of $Q$ forms a symmetric subzonotope in $Z$. It has the bottom
$X$, the color set $T^\cup:=T\cup T^\circ=T'\{j,j^\circ\}$ and is isomorphic to
$Z(\{\xi_i\colon i\in T^\cup\})\simeq Z(d+2,d)$). Using terminology as in
Sect.~\SEC{symm-r-odd}, we call $\frakB$ a \emph{barrel} in $Q$, and refer to
the set of $\binom{d+2}{d}$ cubes of $Q$ occurring in it as the \emph{filling}
of $\frakB$ (where $2d+1$ cubes belong to $\frakD\cup\frakD^\ast$).
\medskip

\noindent\textbf{Definition.} Let $\frakD,\frakD^\ast,\frakB$ be as above. The
\emph{lowering barrel flip} in $Q$ w.r.t. $\frakB$ replaces the filling of
$\frakB$ by the corresponding color-symmetric filling. More precisely, each
face $(X\cup Y\,|\,S)$ of $Q$ within $\frakB$ turns into the face $(X\cup
Y^\circ\,|\,S^\circ)$ (as though making the mirror-reflection w.r.t. the
corresponding hyperplane through $X$).
 \medskip

Under this flip, we obtain again a symmetric cubillage on $Z$. Here the capsid
$\frakD=(X\,|\,T^\circ)$ (having anti-standard filling in $Q$) turns into the
color-symmetric capsid $(X\,|\,T^\circ)$ with the standard filling, and
similarly the capsid $\frakD^\ast=((XT)^\ast\,|\,T^\circ)$ turns into
$((XT)^\ast\,|\,T)$ with the standard filling either.

Thus, flips of both sorts decrease the number of capsids with the anti-standard
filling. Let $\Dscr^+(Q)$ denote the set of such capsids in $Q$. We conjecture
the following.
  \medskip

 \begin{itemize}
\item[(C2)]
Let $n,d$ be odd and let $Q\in\bfQsym_{n,d}$ be such that $\Dscr^+(Q)\ne
\emptyset$. Then there exists a capsid $\frakD\in\Dscr^+(Q)$ such that either
$\frakD,\frakD^\ast$ have no cube in common, or $\frakD,\frakD^\ast$ share a
cube and $\frakD\cup\frakD^\ast$ is extended to a barrel in $Q$ (so $Q$ admits
a lowering double flip in the former case, and a lowering barrel flip in the
latter case).
  \end{itemize}

In light of reasonings above, the validity of (C2) would imply the following
result: for $n,d$ odd, any cubillage in $\bfQsym_{n,d}$ can be connected by a
series of symmetric lowering (double or barrel) flips to the standard cubillage
on $Z(n,d)$ (which is symmetric), yielding the connectedness of
$\bfSsym_{n,d-1}$ via symmetric flips.
\medskip

\noindent\textbf{Remark 4.} For $d$ odd, a symmetric cubillage $Q$ on
$Z(d+2,d)$ can be lifted as a symmetric abstract membrane $M$ in the zonotope
$Z'=Z(d+2,d+1)$. However, $M$ cannot be extended to a \emph{symmetric}
cubillage on $Z'$. Indeed, $Z'$ has exactly two cubillages, standard and
anti-standard ones. They are projections of the  front and rear sides of the
cube $C=Z(n+2,n+2)$, but neither $\Cfr$ nor $\Crear$ is symmetric.


\section{Appendix 3: Symmetric weak $r$-separation.} \label{sec:weak-r-sep}

In Sect.~\SEC{wcoll} we explained how to devise symmetric flips in maximal
symmetric weakly separated collections (or w-collections) in $2^{[n]}$. The
notion of weak separation is generalized in~\cite{DKK5} to any \emph{odd}
integer $r>0$, where
  \begin{numitem1} \label{eq:weakrsep}
sets $A,B\subseteq [n]$ are called \emph{weakly $r$-separated} if they are
strongly $(r+1)$-separated and, in addition: if there are elements $i_1<
i_2<\cdots<i_{r+2}$ of $[n]$ alternating in $A-B$ and $B-A$, then $|A|\le|B|$
when $A$ surrounds $B$ (equivalently, $i_1,i_{r+2}\in A-B$), and $|B|\le|A|$
when $B$ surrounds $A$.
  \end{numitem1}

Accordingly, a collection $\Wscr\subseteq 2^{[n]}$ is called weakly separated
if any two sets in it are such. When $r=1$, this turns into the notion of
w-collection. An important fact shown in~\cite[Th.~1.1]{DKK5} is that (for $r$
odd) the maximal possible sizes of weakly and strongly $r$-separated
collections in $2^{[n]}$ are the same, denoted as $s_{n,r}$. (Note that to
introduce and study the concept of weak $r$-separation when $r$ is even is a
more sophisticated task; see a discussion in~\cite[Appendix~B]{DKK5}).

In what follows we assume that the number $n$ of colors is \emph{even} (while
$r$ is odd) and denote the set (class) of symmetric weakly $r$-separated
collections $\Wscr\subseteq 2^{[n]}$ whose size $|\Wscr|$ is equal to $s_{n,r}$
by $\bfWsym_{n,r}$. (It should be noted that when $n$ is odd, the maximal size
of a symmetric weakly $r$-separated collection in $2^{[n]}$ need not be equal
to $s_{n,r}$; this is seen already for $r=1$ in Sect.~\SSEC{sw-n-odd}. This
case is omitted here.)

Our approach to devise flips symmetric in $\bfWsym_{n,r}$ is based on the
following result (which in turn is a generalization of a result
in~\cite[Th.~1.7]{LZ} on flips ``in the presence of four witnesses'' for usual
w-collections).

\begin{theorem}[see~\cite{DKK5}] \label{tm:neighbor}
For $r$ odd (and $n$ arbitrary) and for $r':=(r+1)/2$, let
$P=(p_1,\ldots,p_{r'})$ and $Q=\{q_0,\ldots,q_{r'}\}$ consist of elements of
$[n]$ such that $q_0<p_1<q_1< \ldots<p_{r'}<q_{r'}$, and let
$X\subseteq[n]-(P\cup Q)$. Define the sets of  \emph{neighbors} (or
``witnesses'') of $P,Q$ to be $\Nscr(P,Q):=\Nscr^\uparrow_{P,Q}\cup
\Nscr^\downarrow_{P,Q}$, where
  \begin{align}
  \Nscr^\uparrow_{P,Q} &:= \{Pq\colon q\in Q\}\cup \{(P-p)q\colon
  p\in P,\, q\in Q\};& \mbox{and} \nonumber\\
   \Nscr^\downarrow_{P,Q} &:= \{Q-q\colon q\in Q\}\cup \{(Q-q)p\colon
  p\in P,\, q\in Q\}.& \nonumber
  \end{align}
Suppose that a weakly $r$-separated collection $\Wscr\subset 2^{[n]}$ contains
the set $XP=X\cup P$ (resp. $XQ$) and the sets $XS$ for all $S\in\
\Nscr_{P,Q}$. Then the collection obtained from $\Wscr$ by replacing $XP$ by
$XQ$ (resp. replacing $XQ$ by $XP$) is again weakly $r$-separated.
  \end{theorem}

(In fact, Theorem~1.2 in~\cite{DKK5} gives a sharper assertion, but this is not
needed to us.) Note that $\{P,Q\}$ is the only pair in
$\Nscr^+_{P,Q}:=\Nscr_{P,Q}\cup\{P,Q\}$ which is not weakly $r$-separated.

We refer to the replacement $XP\leadsto XQ$ (resp. $XQ\leadsto XP$) as the
\emph{raising} (resp. \emph{lowering}) flip using the \emph{gadget}
$\Gscr=\Gscr_{X,P,Q}:= (X\,|\, \Nscr^+_{P,Q})$. Also we say that $\Gscr$ has
the \emph{root} $X$, \emph{type} $T:=P\cup Q$,  \emph{height} $|XP|=|X|+r'$,
\emph{lower layer} $\Llow_{X,T}:=\{XS\colon S\in\Nscr^+_{P,Q},\, |S|=r'\}$, and
\emph{upper layer} $\Lup_{X,T}:=\{XS\colon S\in\Nscr^+_{P,Q},\, |S|=r+1\}$.
\medskip

Now let $\Wscr$ be symmetric and suppose that it contains the gadget $\Gscr$ as
above. Then $\Wscr$ contains the gadget $\Gscr^\ast$ symmetric to $\Gscr$. One
can check that $\Gscr^\ast$ has the root $(XT)^\ast$ and type $T^\circ$; the
latter is partitioned into the alternating sets $P^\circ$ (of size $r'$) and
$Q^\circ$ (of size $r'+1$). The lower layer $\Llow_{(XT)^\ast,T^\circ}$ of
$\Gscr^\ast$ is symmetric to $\Lup_{X,T}$, and the upper layer
$\Lup_{(XT)^\ast,T^\circ}$ to $\Llow_{X,T}$. Then (in view of~\refeq{AAn})
  \begin{numitem1} \label{eq:GGast}
the height $|(XT)^\ast|+r'$ of $\Gscr^\ast$ is equal to $n-1+(|X|+r')$, the set
$(XT)^\ast P^\circ$ is symmetric to $XQ$, and $(XT)^\ast Q^\circ$ is symmetric
to $XP$.
  \end{numitem1}

\noindent\textbf{Definition.} The symmetric flip in $\Wscr$ using a gadget
$\Gscr=\Gscr_{X,P,Q}$ (and its symmetric gadget $\Gscr^\ast$) consists of the
raising flip w.r.t. one, and the lowering flip w.r.t. the other of $\Gscr$ and
$\Gscr^\ast$, say, $XP\leadsto XQ$ and $(XT)^\ast Q^\circ\leadsto (XT)^\ast
P^\circ$.
  \medskip

A reasonable question is whether these two flips are compatible. This is so,
and the resulting double flip produces again a symmetric weakly $r$-separated
collection, if the gadgets $\Gscr$ and $\Gscr^\ast$ (regarded as subcollections
in $\Wscr$) do not meet. The latter is guaranteed when the difference of their
heights $h_\Gscr:=|X|+r$ and $h_{\Gscr^\ast}:= |(XT)^\ast|+r'$ is greater than
or equal to 2.

Suppose that $\Delta:=|h_\Gscr-h_{\Gscr^\ast}|<2$. Since $n$ is even,
\refeq{GGast} implies that $\Delta=1$. Then, w.l.o.g., we may assume that
$h_{\Gscr^\ast}=h_\Gscr+1$. In other words, the heights of the upper layer of
$\Gscr$ and the lower layer of $\Gscr^\ast$ are equal. Nevertheless, in this
case, if the transformation consists of the raising flip in $\Gscr$ and the
lowering flip in $\Gscr^\ast$, then no conflict can arise (since the former
flip preserves both layers of $\Gscr^\ast$).

On the other hand, the lowering flip $XQ\leadsto XP$ in $\Gscr$ may affect
$\Gscr^\ast$. This happens if $XQ$ belongs to the lower layer of $\Gscr^\ast$,
in which case $XQ$ disappears, the lower layer of $\Gscr^\ast$ decreases, and
we cannot appeal to Theorem~\ref{tm:neighbor} with $\Gscr^\ast$. We conjecture
that

  \begin{itemize}
\item[(C3):]
If $h_{\Gscr^\ast}=h_\Gscr+1$, then the set $XQ$ does not belong to
$\Llow_{(XT)^\ast,P^\circ}$.
  \end{itemize}

Subject to the validity of~(C3), it is reasonable to raise the next conjecture:

  \begin{itemize}
\item[(C4):]
For $r$ odd and $n$ even, any two collections in $\bfWsym_{n,r}$ can be
connected by a series of symmetric flips using gadgets as above.
  \end{itemize}

\noindent\textbf{Remark 5.} We know (cf.~\refeq{cub-dsepar}) that any
size-maximal strongly $r$-separated collection in $2^{[n]}$ is representable,
i.e., it is viewed as the vertex set of a cubillage on $Z(n,r+1)$ or,
equivalently, of a (strong) membrane of some cubillage on $Z(n,r+2)$. In
contrast, it is open at present, whether any size-maximal weakly $r$-separated
collection in $2^{[n]}$ is representable, in the sense that it forms the vertex
set of a \emph{weak membrane} in the \emph{fragmentation} of some cubillage on
$Z(n,r+2)$  (for definitions, see~\cite[Sec.~6]{DKK5}, where also the above
open question is stated as a conjecture). In light of this, one can simplify
verification of~(C3) and weaken~(C4), by restricting ourselves by the
(sub)class of representable collections in $\bfWsym_{n,r}$; this would enable
us to use a geometric interpretation of gadgets (which are associated with some
``central fragments'' of cubes of odd dimensions).

\end{document}